\documentclass[11pt,leqno,psfig,openbib]{amsart}
\usepackage{amssymb,amsmath,graphicx,amsfonts,euscript}
\usepackage{fancyhdr}
\usepackage{epsfig}
\usepackage{setspace}
\usepackage{dsfont}
\usepackage[red]{xcolor}
\usepackage{vmargin}
\usepackage{bbm}

\if@secthm
 \newtheorem{thm}{Theorem}[section]
\else
 \newtheorem{thm}{Theorem}
\fi
\newtheorem{rem}{Remark}[section]

\newtheorem{prop}{Proposition}[section]

\numberwithin{equation}{section}
\numberwithin{thm}{section}
\title{Global existence for fully nonlinear reaction-diffusion systems 
describing multicomponent reactive flows}
\author{Martine Marion$^1$ and Roger Temam$^2$}

\address{$^1$ MM:  Universit\'e de Lyon, Ecole Centrale de Lyon, CNRS UMR 5208, DMI, 36 avenue Guy de Collongue, 69134 Ecully Cedex, France.}
\email[Martine Marion]{Martine.Marion@ec-lyon.fr}
\address{$^2$ RT:  The Institute for Scientific Computing and Applied Mathematics, Indiana University, 831 East Third Street, Bloomington, Indiana  47405, U.S.A.}
\email[Roger Temam]{temam@indiana.edu}


\begin{document}
\pagestyle{plain}
\maketitle
%\centerline{\textit{Dedicated to the memory of Oscar Manley}}
\baselineskip=14truept

\begin{abstract}
We consider combustion problems in the presence of complex chemistry and nonlinear diffusion laws leading to fully nonlinear multispecies reaction-diffusion equations.  We establish results of existence of solution and maximum principle, i.e. positivity of the mass fractions, which rely on specific properties of the models.  The nonlinear diffusion coefficients are obtained by resolution of the so-called Stefan-Maxwell equations.
\end{abstract}

%\thispagestyle{fancy} \headheight 15pt

%\rhead[]{\textit{In preparation.}}

\section{Introduction}\label{s1}

In this paper we investigate some mathematical issues arising in the context of the coupling of multi-species exothermic chemical reactions to fluid motion.  The physical paradigm for
this problem is combustion.  Another related important problem is that of multi-species
endothermic chemical reactions, with applications for instance to the chemistry of the high atmosphere; this problem will be studied elsewhere, and we concentrate here on
exothermic chemical reactions and combustion.

Mathematical models for multi-species chemical reactions almost exclusively deal with the special case of chemical species whose binary diffusion coefficients are constants all equal to one another.  For it is only in that
case that the coefficients of the Laplacians in the reaction-diffusion equations are simply those diffusion constants; see for instance \cite{MMT93} and the references therein.

In the present article we are concerned with the more general case, more physically relevant, for which the binary coefficients differ from pair to pair, the constraint of momentum conservation, i.e.
 the vanishing of the sum of diffusion fluxes, leading to inescapable nonlinear coefficients associated with the second spatial derivatives in the reaction-diffusion equations governing
 the evolution of the chemical species.  The situation is further complicated by the fact that here the linear relationship between the diffusion velocities of the various species and
 the concentration gradients of those species is given by the resolution of a singular linear system expressing the so-called Stefan-Maxwell equations \cite{Max67}, \cite{Ste71}, \cite{BSL07}, \cite{Wil88}.

As a result, it is not all clear if the equations governing the evolution of the various chemical species, yield solutions that are physically meaningful as well as mathematically sound.
Such questions as boundedness, positive invariance and existence deserve to be addressed.  It is that which is the subject of this article.  The connection with the motion induced by
these more general exothermic reactions is examined as well.

\vskip0.1in

From the mathematical viewpoint the system of equations that we consider in space dimension $n=2$ or $3$ consists of the following:

- the Navier-Stokes equations for incompressible fluid corresponding to pressure and velocity $p,\boldsymbol{v},$ for the mixture,

- the heat equation for the temperature $\theta$ with a heat source term corresponding to the Arrhenius law,

- the evolution (conservation) equations for the mass fractions $Y_1,\ldots ,Y_N,$ of the $N$ species.

%\begin{itemize}
%\item{} the Navier-Stokes equations for incompressible fluid corresponding to pressure and velocity $p,\boldsymbol{v},$ for the mixture,
%\item{} the heat equation for the temperature $\theta$ with a heat source term corresponding to the Arrhenius law,
%\item{} the evolution (conservation) equations for the mass fractions $Y_1,\ldots ,Y_N,$ of the $N$ species.
%\end{itemize}

The boundary value problems that we study correspond to reasonable boundary conditions for a flame propagating upward in a vertical tube but it
is clear that other related boundary value problems can be studied by similar methods.

As indicated before, the diffusion terms in these equations are nonlinear; for each of these equations it is a combination of $\boldsymbol{\nabla} Y_1,\dots ,\boldsymbol{\nabla} Y_N$, with coefficients
rational functions of $Y_1, \ldots ,Y_N.$  These functions are not given explicitly; they are instead given by the resolution of the Stefan-Maxwell equations.  In  Section 3 we derive
enough information on these coefficients to be able to conduct our theoretical study.  The first rigorous mathematical study of the Maxwell-Stefan linear system can be found, to the best of our knowledge, in \cite{Gio90}, \cite{Gio91} which mainly address questions of numerical computations ; see also 
 \cite{EG94}, \cite{EG97},  \cite{Gio99}, \cite{Lar91} for the numerical computation of the diffusion coefficients and \cite{WT62}, \cite{WU70} for the kinetic theory background. 

%{\color{blue}The main emphasis in this article is on the reaction-diffusion equations for the mass fractions.  Existence of solutions is based on an a priori estimate resulting from the
%principles of thermodynamics involving the functions log $Y_i$ (Gibbs energy) \cite{LL75}. ROGER c'est deplace ci-dessous mais il faudrait peut-etre un peu faire mieux}

The article is organized as follows.  In Section 2 we describe the equations and the initial and boundary value problems and state the main results for the chemistry equations and for the
complete system corresponding to the coupling with hydrodynamics.  In Section 3 we study in details the Stefan-Maxwell equations considered as a singular linear algebraic system
for the diffusion velocities $\boldsymbol{V}_1, \ldots \boldsymbol{V}_N$ or the corresponding fluxes $\boldsymbol{F}_i$. We show there how to determine the diffusion fluxes $\boldsymbol{F}_i$ in terms of the mass fractions $Y_i$ and their gradients. These fluxes become singular when all $Y_i$ vanish, a case that it is necessary to handle in our mathematical investigation. A crucial tool in our approach is to  define modified diffusion coefficients that yield the proper fluxes for the actual solutions of the Stefan-Maxwell diffusion equations and that remain regular when all the mass fractions $Y_i$ tend to zero. Also we obtain enough information on the fluxes for our purpose and in particular to infer energy estimates from the equation for the Gibbs energy (see below). We conclude this section with explicit calculations for the relevant and interesting case of three species.
In Section 4  we prove the results previously stated for the reaction-diffusion equations alone, assuming that the velocity and temperature are given.
 For that purpose we approximate the equations by more regular ones ; these are equations for all mass fractions $Y_i$ treated as formally independent unknowns for which the positivity conditions are not imposed.  Afterwards we deduce that $Y_i\geq 0$ by using the maximum principle and show that
$\sum^N_{j=1} Y_j=1$. In order to pass to the limit, we then proceed with the fundamental energy estimate that results from the
principles of thermodynamics involving the functions log $Y_i$ (Gibbs energy) \cite{LL75}. This step requires a detailed study in particular due to the singularities in the log $Y_i$ - terms. Our estimate allows us to pass to the limit, solving the exact equations.  In Section 5 we couple the chemistry equations with the fluid and heat equations; we prove
the existence result for the complete (coupled) system using the same method of regularization.

The main results in this article were announced in the note \cite{MMT95}, and a draft was written which was not completed at that time.  After the passing away of Oscar Manley
in 2001, the two others authors regained interest in this work in relation with recent developments on the subject, (see e.g. \cite{Bot11}, \cite{BGS12}, \cite{JS13} and the references therein), and with possible applications to the chemistry of the atmosphere.  Additional noteworthy applications are listed in \cite{JS13}.

Concerning the mathematical analysis of the diffusion partial differential equations, local in time results can be found in  \cite{GM98a}, \cite{Bot11} while particular cases are considered in \cite{BGS12} ,\cite{GM98b} and \cite{Bot11}.  The general case is considered in \cite{JS13} where the existence  of solutions is derived for all time.  In fact in \cite{JS13} the results do not pertain to the usual (classical) system that we consider but to a formally equivalent system obtained in particular by {\em  assuming} that  $Y_i>0$ at all time.  Furthermore in \cite{JS13} the quantity that we call $Y_M$ below, $Y_M=\sum^N_{i=1}Y_i/M_i$ is required to be constant.  This assumption is licit when considering the isobaric isothermal case as done in \cite{JS13} but 
 not when coupling with hydrodynamics and combustion as we do here.  Finally, in the approach of \cite{JS13}, the symmetry between the mass fractions $Y_1,\ldots ,Y_N$ is broken by taking advantage of the relation $\sum^N_{i=1} Y_i=1$ and eliminating one of the mass fractions, and other changes of variables are performed.  Doing so the authors lose several structural properties of the system including the maximum principle for the mass fractions. On the contrary, a key point in our approach is to keep all the mass fractions, thus keeping the symmetry between the unknowns $Y_1,\ldots ,Y_N.$

\vskip0.1in

{\em This article is dedicated to the memory of Oscar Manley who suggested this work and who was actively involved in it, with kind memories and our great appreciation for his scientific vision and his tremendous scientific culture.}

%\newpage
\vskip0.1in

\section{The Equations and the Main Results}\label{s2}
\subsection{Description of the problem}\label{ss2.1}\hspace*{\fill} \\
We consider a multi-component premixed gas flame propagating in a bounded channel $\Omega\subset\mathbb{R}^n$, $n = 2$ or $3$.  We assume that $\Omega = (0,\ell)\times (0,h)$ if $n=2$ and
$\Omega = (0,\ell) \times (0,L) \times (0,h)$ if $n=3.$  We denote by $x=(x_1, x_2)$ or $(x_1,x_2,x_3)$ a generic point in $\mathbb{R}^2$ or $\mathbb{R}^3$ while
$\left\{\textbf{e}_1, \textbf{e}_2\right\}$ or $\left\{\textbf{e}_1, \textbf{e}_2,\textbf{e}_3\right\}$ denotes the canonical orthonormal basis where $\textbf{e}_n$
is parallel to the ascending vertical.  Under suitable assumptions (see \cite{Wil88} or \cite{MMT93}), and in particular assuming that the fluid is incompressible and using the Boussinescq approximation,
the equations for the reactive flow read
\begin{equation}\label{e2.1}
\displaystyle \frac{\partial\boldsymbol{v}}{\partial t} + (\boldsymbol{v}\cdot\boldsymbol{\nabla})\boldsymbol{v} - Pr \boldsymbol{\Delta}\boldsymbol{v} + \boldsymbol{\nabla} p = \textbf{e}_n\sigma\theta,
\end{equation}
\begin{equation}\label{e2.2}
div\enspace\boldsymbol{v} = 0,
\end{equation}
\begin{equation}\label{e2.3}
\displaystyle\frac{\partial\theta}{\partial t} + (\boldsymbol{v} \cdot\boldsymbol{\nabla})\theta - \Delta\theta = - \sum^N_{i=1} h_i\omega_i(\theta, Y_1, \ldots , Y_N),
\end{equation}
\begin{equation}\label{e2.4}
\displaystyle\frac{\partial Y_i}{\partial t} + (\boldsymbol{v}\cdot\boldsymbol{\nabla})Y_i +\boldsymbol{\nabla}\cdot \textbf{F}_i = \omega_i(\theta, Y_1, \ldots , Y_N),\enspace 1\leq i\leq N.
\end{equation}

The unknowns, which are here in non-dimensional form, are the velocity $\boldsymbol{v} = (v_1,v_2)$ or $(v_1,v_2,v_3)$, the pressure $p,$ the temperature $\theta$ and the mass fractions $Y_i$
of the $N$ species involved in the chemical reactions.  Furthermore $h_i, \sigma$ and $Pr$ (the Prandtl number) are positive constants.  The structure of the $\omega_i$ which are given
functions of $\theta,Y_1,\ldots, Y_N, $
is described below, in
\eqref{e2.15}-\eqref{e2.19}.  Naturally the mass fractions $Y_i$ are expected to satisfy the conditions
\begin{equation*}
Y_i\geq 0 \text{ for } 1\leq i\leq N,\quad \sum^N_{i=1} Y_i=1.
\end{equation*}

We now discuss the form of the fluxes $\textbf{F}_i.$  Our purpose is to study this problem in the case of complex multi-component diffusion laws.  The fluxes $\textbf{F}_i$ in \eqref{e2.4} read
\begin{equation}\label{e2.5}
\textbf{F}_i = Y_i\textbf{V}_i,
\end{equation}
where $\textbf{V}_i$ is the diffusion velocity of species $i,$ so that
\begin{equation}\label{e2.6}
\sum^N_{i=1}Y_i\textbf{V}_i = 0.
\end{equation}

Under general assumptions, the diffusion velocities are given (implicitly) in terms of the gradients of the mole fractions $X_i$ by the Stefan-Maxwell equations (see \cite{Max67}, \cite{Ste71}, \cite{BSL07}, \cite{Wil88}):
\begin{equation}\label{e2.7}
\boldsymbol{\nabla} X_i = \sum^N_{j=1, j\not= i} d_{ij}X_iX_j (\textbf{V}_j - \textbf{V}_i),\quad i= 1, \ldots , N,
\end{equation}
where $d_{ij} = \kappa/D_{ij}$ and $D_{ij} = D_{ji} > 0$ is the binary diffusion coefficient for species $i$ and $j$ while
$\kappa$ is the thermal diffusion coefficient, here taken to be a constant.  The resolution of \eqref{e2.7} is not straightforward
since this linear system (with respect to the $\textbf{V}_i$) has a singular matrix.  Also, these equations involve the $X_i$ while
equations \eqref{e2.4} concern the mass fractions $Y_i.$  The algebraic relations between the $X_i$ and the $Y_i$ are given in
\eqref{e2.14} below and in section \ref{s3} where we conduct a detailed study of the resolution of the Stefan-Maxwell equations.  It is
found there that the fluxes $\textbf{F}_i$ can be defined for arbitrary smooth (say $C^1$) functions $Y_i$ from $\Omega$ into $[0,+\infty)$, and
have the form
\begin{equation}\label{e2.8}
\textbf{F}_i = - \sum^N_{j=1} a_{ij}(Y_1,\ldots , Y_N)\boldsymbol{\nabla} Y_j, \quad\text{ for }i=1,\ldots ,N,
\end{equation}
with
\begin{equation}\label{e2.9}
\sum^N_{i,j=1}a_{ij}(Y_1,\ldots ,Y_N)\boldsymbol{\nabla} Y_j=0.\footnote{Property \eqref{e2.9} is valid even if $\sum^N_{j=1} Y_j \not= 1;$ see (\ref{e3.48}) in Section \ref{s3}.}
\end{equation}
The coefficients $a_{ij}$ are rational functions of $Y_1,\ldots ,Y_N$, continuous from $[0,+\infty)^N$ into $\mathbb{R}$ and such that, for
$i,j=1, \ldots ,N$,
\begin{equation}\label{e2.10}
\begin{array}{cc}
a_{ij}(Y_1, \ldots ,Y_N) = Y_ib_{ij}(Y_1,\ldots ,Y_N), \text{ for }i\not=j,\\
\\
 \text{ where }b_{ij}:[0,+\infty)^N\rightarrow\mathbb{R}\text{ is continuous, }
\end{array}
\end{equation}
\begin{equation}\label{e2.11}
\begin{array}{cc}
a_{ii}(Y_1,\ldots ,Y_N) = b^0_i(Y_1, \ldots ,Y_N) + Y_ib^1_i(Y_1,\ldots ,Y_N),\\
\\
\text{ where }b^0_i\text{ and }b^1_i:[0,+\infty)^N\rightarrow\mathbb{R}\text{ are continuous and }b^0_i(Y_1,\ldots , Y_N)\geq 0.
\end{array}
\end{equation}
\noindent Also the following property is proved to hold:  there exists a constant $c_1>0$ such that
\begin{equation}\label{e2.12}
\begin{array}{cc}
\text{if } Y_1,\ldots Y_N\in H^1(\Omega)\text{ are such that }0\leq Y_i(x)\leq 1\text{ and }\\
\\
\sum^N_{j=1} Y_j(x) = 1\text{ for a.e. }x\in\Omega,\text{ then }\\
\\
-\sum^N_{i=1}\textbf{F}_i\cdot\boldsymbol{\nabla}\mu_i\mathbbm{1}_{(Y_i>0)}\geq c_1\sum^N_{i=1}|\boldsymbol{\nabla} Y_i|^2,\text{ for a.e. }x\in\Omega,
\end{array}
\end{equation}
where $\mathbbm{1}_{(Y_i>0)}$ is the characteristic function of the set $\left\{x\in\Omega,\; Y_i(x)>0\right\}$ and $\mu_i =\mu_i(x)$ is only defined where $Y_i(x)>0$ (or equivalently $X_i(x)>0$) by:
\begin{equation}\label{e2.13}
\mu_i=\frac{1}{M_i}\log X_i,\text{ if }Y_i>0.
\end{equation}
Here $X_i$ is the mole fraction of species $i$ given by
\begin{equation}\label{e2.14}
X_i =\frac{Y_i}{M_iY_M} , \enspace Y_M = \sum^N_{j=1}\frac{Y_j}{M_j},\enspace M_j = \text{ molecular mass of species }j,
\end{equation}
and $\boldsymbol{\nabla}\mu_i$ is defined almost everywhere when $Y_i$ (or $X_i )> 0$ by $\boldsymbol{\nabla}\mu_i(x)=\boldsymbol{\nabla} X_i(x)/M_iX_i(x).$

We will study equations \eqref{e2.1}-\eqref{e2.4} using the above properties of the fluxes $\textbf{F}_i.$  We show in Section \ref{s3} how the properties
\eqref{e2.8}-\eqref{e2.12} can be actually proved for the fluxes $\textbf{F}_i$ given by \eqref{e2.5}-\eqref{e2.7}, or more precisely for suitably modified fluxes.

We now state the assumptions on the chemical rates $\omega_{i};\enspace \omega_i$ is the difference between the rate of production of species $i,\alpha_i =\alpha_i (\theta, Y_1, \ldots , Y_N)\geq 0,$ and the rate of removal of species $i$ ; the rate of removal of species $i$ is proportional to an integral power of $Y_i$ and we write it in the form $Y_i\beta_i(\theta, Y_1\ldots ,Y_N),$ with $\beta_i\geq 0.$  Hence:
\begin{equation}\label{e2.15}
\omega_i =\omega_i(\theta, Y_1,\ldots ,Y_N)= \alpha_i(\theta, Y_1, \ldots , Y_N) - Y_i\beta_i(\theta ,Y_1, \ldots , Y_N).
\end{equation}
We assume that the functions $\alpha_i$ and $\beta_i$ are defined for $\theta\geq 0$ and $0\leq Y_k\leq 1,$ are continuous on $\mathbb{R}_+\times [0,1]^N$ and that
\begin{align}\label{e2.16}
&\alpha_i(\theta, Y_1, \ldots , Y_N)\geq 0,\enspace \beta_i(\theta , Y_1,\ldots , Y_N)\geq 0\enspace \text{ for }\theta\geq 0,\enspace 0\leq Y_k\leq 1,\\
\label{e2.17}
&\sum^N_{i=1}\omega_i(\theta ,Y_1, \ldots , Y_N) = 0, \text{ for }\theta\geq 0, \enspace 0\leq Y_k\leq 1,\\
\label{e2.18}
&\alpha_i,\;\beta_i\text{ and hence } \omega_i\text{ are bounded on }[0,+\infty) \times [0,1]^N,\\
\label{e2.19}
&\sum^N_{i=1} h_i\omega_i(0,Y_1,\ldots ,Y_N)\leq 0,\enspace \text{ for }0\leq Y_k\leq 1.
\end{align}
Note that these abstract assumptions are satisfied by the rates given by the Arrhenius law.  See \cite{MMT93} for specific examples.
\vskip0.1in
Equations \eqref{e2.1}-\eqref{e2.4} are supplemented with appropriate boundary and initial conditions.  We have set
$\Omega = (0,\ell) \times  (0,h)$ for
$n=2$ and $\Omega = (0,\ell) \times  (0,L)\times (0,h)$ for $n=3.$  We assume that the flame propagates in the vertical $x_n$ direction, the premixed reacting species entering from below.  The vertical sides of the channel are adiabatically insulated and impervious to fluid flow.  We denote by $\Gamma_0$ and $\Gamma_h$ the parts of the boundary $\partial\Omega$ of $\Omega$ corresponding to $x_n=0$ and $x_n=h$ and we denote by $\Gamma_\ell$ the lateral boundary corresponding to $0<x_n<h.$ Consequently, the boundary conditions read
\begin{equation}\label{e2.20}
v_i=0\text{ on } \partial\Omega \text{ for } 1\leq i\leq n-1,  \enspace v_n=1\text{ on }\Gamma_0\cup\Gamma_h,  \enspace \frac{\partial v_n}{\partial\boldsymbol{\nu}} = 0\text{ on }\Gamma_\ell,
\end{equation}
\begin{equation}\label{e2.21}
\theta = 0\text{ on }\Gamma_0,\enspace 
\frac{\partial\theta}{\partial\boldsymbol{\nu}} = 0\text{ on }\Gamma_h\cup\Gamma_\ell,
\end{equation}
and for $1 \leq i \leq N$ :
\begin{equation}\label{e2.22}
\begin{cases}
&Y_i = Y^u_i\text{ on }\Gamma_0,\\
&\boldsymbol{\nu}\cdot\textbf{F}_i = 0 \text{ on }\Gamma_h\cup\Gamma_\ell,
\end{cases}
\end{equation}
that is for \eqref{e2.22}$_2$:
\begin{equation}\label{e2.22a}
\left(\sum^N_{j=1}a_{ij}(Y_1,\ldots ,Y_N)\boldsymbol{\nabla} Y_j\right)\cdot\boldsymbol{\nu}= 0\text{ on }\Gamma_h\cup\Gamma_\ell.
\end{equation}
Here $\boldsymbol{\nu} = (\nu_1\ldots ,\nu_n)$ is the unit outward normal on $\partial\Omega$ and $Y^u_i$, $1\leq i\leq N,$ is the concentration of the species $Y_i$ as it enters the channel (unburnt gas).  The $Y^u_i$ are assumed to be constant and satisfy
\begin{equation}\label{e2.23}
Y^u_i>0\enspace\forall i,\enspace \enspace\sum^N_{i=1}Y^u_i = 1.
\end{equation}
\vskip0.1in

Finally, we associate with \eqref{e2.1}-\eqref{e2.4} and \eqref{e2.20}-\eqref{e2.22}, the initial conditions
\begin{equation}\label{e2.24}
\boldsymbol{v} (x,0) = \boldsymbol{v}_0(x),\enspace\theta (x,0) =\theta_0(x),
\end{equation}
\begin{equation}\label{e2.25}
Y_i(x,0) = Y_{i,0}(x),
\end{equation}
where we assume that
\begin{equation}\label{e2.26}
\theta_0(x)\geq 0,
\end{equation}
\begin{equation}\label{e2.27}
Y_{i,0}(x)\geq 0,\quad\sum^N_{i=1} Y_{i,0}(x)=1.
\end{equation}

\subsection{Existence results}\label{ss2.2} \hspace*{\fill} \\
To state our existence results it is convenient to extend the domain of definition of the reaction rates $\omega_i$, $1\leq i\leq N$, to $\mathbb{R}^{N+1}$ by setting
\begin{equation}\label{e4.2}
\omega_i(\theta, Y_1,\ldots, Y_N) = \omega_i(\theta^+, \psi (Y_1), \ldots, \psi(Y_N)),\quad\theta\in\mathbb{R}, \quad Y_k\in\mathbb{R},
\end{equation}
where, for $s\in\mathbbm{R}$, $s^+ = \max (s,0)$ and:
$$\psi(s) = s\text{ if }0\leq s\leq 1, \psi(s) =  1\text{ if }s\geq 1,\psi(s) =  0\text{ if }s\leq 0.$$

We first consider the system \eqref{e2.4}, assuming that $\boldsymbol{v}$ and $\theta$ are given such that, for some $T>0$:
\begin{equation}\label{e2.28}
\begin{cases}
&\boldsymbol{v}\in L^\infty(0, T; L^2(\Omega)^n)\cap L^2(0,T; H^1(\Omega)^n),\\
&\boldsymbol{v}\text{ satisfies }\eqref{e2.2} \text{ and the Dirichlet boundary conditions in } \eqref{e2.20}.
\end{cases}
\end{equation}
\begin{equation}\label{e2.29}
\theta\in L^\infty(0,T;L^2(\Omega))\cap L^2(0,T;H^1(\Omega)).
\end{equation}
\vskip0.1in

The following existence result holds.
\begin{thm}\label{t2.1}
Under the assumptions \eqref{e2.8}-\eqref{e2.12}, \eqref{e2.15}-\eqref{e2.18}, \eqref{e2.23}, let $\boldsymbol{Y}_0=(Y_{i,0})_{1\leq i\leq N}$ be given in
$L^2(\Omega)^N$ such that \eqref{e2.27} holds for almost every $x$ in $\Omega,$ and let $\boldsymbol{v}$ and $\theta$ be given satisfying \eqref{e2.28} and
\eqref{e2.29}.  Then, problem \eqref{e2.4}, \eqref{e2.22}, \eqref{e2.25} possesses a solution $\textbf{Y}=(Y_i)_{1\leq i\leq N}$ such that
\begin{equation}\label{e2.30}
\boldsymbol{Y}\in L^\infty(0,T; L^2 (\Omega)^N)\cap L^2(0,T; H^1(\Omega)^N).\\
\end{equation}
Furthermore, we have
\begin{equation}\label{e2.31}
0\leq Y_i(x,t)\leq 1 \text{ and } \sum^N_{i=1} Y_i(x,t) = 1,\text{ for } t\in(0,T) \text{ and a.e. }x\in\Omega.
\end{equation}
\end{thm}

\begin{rem}\label{r2.1a}
To be more precise the solution $\boldsymbol{Y}$ in  the theorem \ref{t2.1} is a weak solution that satisfies the following variational formulation for $ 1\leq i\leq N$:
\begin{equation}\label{e2.31a}
\begin{split}
\bigg<\frac{\partial Y_i}{\partial t}, z_i\bigg> +\int_\Omega [(\boldsymbol{v}\cdot\boldsymbol{\nabla} )Y_i]&z_idx+\sum^N_{j=1}\int_\Omega a_{ij}(Y_1,\ldots ,Y_N)
\boldsymbol{\nabla} Y_j\cdot\boldsymbol{\nabla} z_idx\\
&= \int_\Omega \omega_i(\theta, Y_1,\ldots, Y_N)z_idx, \enspace \forall z_i \in H^{1}_{\Gamma_{0}} (\Omega), 
\end{split}
\end{equation}
%\vskip0.1in
where
\begin{equation*}
H^{1}_{\Gamma_{0}} (\Omega) = \left\{ z\in H^{1}(\Omega), \enspace z = 0\text{ at }x_n=0\right\},
\end{equation*}
and $<\cdot ,\cdot >$ denotes the duality product between $H^{1}_{\Gamma_{0}} (\Omega)$ and its dual. We infer from \eqref{e2.31a} that $\boldsymbol{Y}$ satisfies 
\begin{equation*}
\frac{\partial Y_i}{\partial t} \in L^2(0,T; H^{1}_{\Gamma_{0}} (\Omega)'), \text{ for } 1\leq i\leq N,\\
\end{equation*}
which together with \eqref{e2.30} guarantees that $\boldsymbol{Y}\in C([0,T]; L^2 (\Omega)^N)$.
\end{rem}

\vskip0.1in
We now consider the general system \eqref{e2.1}-\eqref{e2.4}.  The following existence result holds:
\begin{thm}\label{t2.2}
In space dimension $n=2$ or $3,$ under the assumptions \eqref{e2.8}-\eqref{e2.12}, \eqref{e2.15}-\eqref{e2.19}, \eqref{e2.23}, let
\begin{equation*}
\boldsymbol{v}_0\in L^2(\Omega)^n,\quad\theta_0\in L^2(\Omega),\quad \boldsymbol{Y}_0=(Y_{i,0})_{1\leq i\leq N}\in L^2(\Omega)^N
\end{equation*}
be given such that \eqref{e2.2}, \eqref{e2.26}, \eqref{e2.27} hold for almost every $x$ in $\Omega$ and
$$\boldsymbol{v}_0\cdot\boldsymbol{\nu} = v_n=1 \text{ on }\Gamma_0\cup\Gamma_h,
\enspace \boldsymbol{v}_0\cdot\boldsymbol{\nu} = 0 \text{ on }\Gamma_\ell.\footnote{Note that these conditions make sense because $div\enspace\boldsymbol{v}_0 = 0$ by \eqref{e2.2}, see  e.g. \cite{Tem77}.}$$
Then, for any $T>0,$ the problem \eqref{e2.1}-\eqref{e2.4}, \eqref{e2.20}-\eqref{e2.22}, \eqref{e2.24}-\eqref{e2.25} possesses a solution
$(\boldsymbol{v},\theta, \boldsymbol{Y})$ such that
\eqref{e2.28}, \eqref{e2.29}, \eqref{e2.30}, \eqref{e2.31} hold 
and
\begin{equation}\label{e2.33}
\theta(x,t)\geq 0 \text{ for } t\in(0,T) \text{ and a.e. }x\in\Omega.
\end{equation}
\end{thm}

\begin{rem}\label{r2.1b}
Again the solution $(\boldsymbol{v},\theta, \boldsymbol{Y})$ given by Theorem \ref{t2.2} is to be understood as a weak solution satisfying a suitable variational formulation. The equations for $\boldsymbol{Y}$ are given by \eqref{e2.31a} while the ones for $\boldsymbol{v}$ and $\theta$ can be writen down in a standard way.
\end{rem}

\begin{rem}\label{r2.1}
The regularity of the solutions and their uniqueness will be investigated in a separate work.  Uniqueness can only be considered in space dimension 2 since, in space dimension 3, we encounter the difficulties of the incompressible Navier-Stokes equations in that space dimension.  For the regularity, in space dimension 2, we immediately obtain from \eqref{e2.1}-\eqref{e2.3}, and \eqref{e2.29}, \eqref{e2.30}, that:
\begin{equation}\label{e2.35}
\begin{cases}
&\boldsymbol{v}\in L^\infty(0,T; H^1 (\Omega)^2)\cap L^2(0,T; H^2(\Omega)^2),\\
&\theta\in L^\infty (0,T;H^1(\Omega))\cap L^2(0,T;H^2(\Omega)),
\end{cases}
\end{equation}
if $\boldsymbol{v}_0 \in H^1(\Omega)^2$ and $\theta_0 \in H^1(\Omega)$ satisfy the Dirichlet boundary conditions in \eqref{e2.20} and \eqref{e2.21} (see e.g. \cite{Tem77}).
\end{rem}

\begin{rem}\label{r2.1c}
From the mathematical point of view, Theorem \ref{t2.1} extends to all dimensions $n$. Theorem  \ref{t2.2} involving the coupling with the Navier-Stokes equations could extend to all dimensions $n$ as well, with some adjustments for the Navier Stokes equations as in \cite{Lio69}, see also \cite{Tem77}.
\end{rem}

\vskip0.1in

\section{The Stefan-Maxwell Equations and their Solution}\label{s3}
In this Section \ref{s3}, we study the fluxes given by \eqref{e2.5}-\eqref{e2.7}.   Also we introduce suitably modified fluxes which satisfy  the properties \eqref{e2.8}-\eqref{e2.12}. Finally we study explicitly the case of three species.

\subsection{The chemical background}\label{ss3.1}\hspace*{\fill} \\
We consider $N$ different chemical species and denote by $M_i$ the molecular mass of species $i$ and by $f_i = f_i(x,\xi, t)$ the velocity distribution function for molecules of species $i.$  Hence
\begin{equation*}
f_i(x,\xi ,t)dxd\xi
\end{equation*}
denotes the probable number of molecules of type $i$ in the range $dx = dx_1\ldots dx_n$ about the spatial position $x\in\mathbb{R}^n$ and with velocities in the range $d\xi = d\xi_1\ldots d\xi_n$ about the velocity $\xi$ at time $t.$

The total number of molecules of kind $i$ per unit spatial volume at $(x,t)$ is denoted by $N_i = N_i(x,t)$:
\begin{equation*}
N_i(x,t)=\int_{\mathbb{R}^n}f_i(x,\xi, t)d\xi ,\enspace i=1, \ldots , N.
\end{equation*}
The molecular concentration of species $i$ is
\begin{equation*}
C_i = N_i/\mathcal{A}, \enspace i =1, \ldots ,N,
\end{equation*}
where $\mathcal{A}$ is the Avogadro number.

The quantities that we will use and study are $\rho_i, Y_i, X_i$ defined as follows:\goodbreak\noindent
 $-\enspace \rho_i$ is the density of species $i$ (mass per unit volume):
\begin{equation}\label{e3.1}
\rho_i = M_iN_i = \mathcal{A}M_iC_i, \enspace i=1,\ldots , N,
\end{equation}
and
\begin{equation*}
\rho = \sum^N_{i=1}\rho_i,
\end{equation*}
is the total density.  We assume incompressibility; hence the total density is constant in space and time
\begin{equation*}
\rho =\rho_0.
\end{equation*}
$-\enspace Y_i$ is the mass fraction of species $i$:
\begin{equation}\label{e3.2}
Y_i = \frac{\rho_i}{\rho}, \enspace i=1, \ldots ,N,
\end{equation}
so that
\begin{equation}\label{e3.3}
0\leq Y_i\leq 1 \text{ for }  1 \leq i \leq N, \quad\sum^N_{i=1}Y_i=1.
\end{equation}
$-\enspace X_i$ is the mole fraction of species $i$:
\begin{equation}\label{e3.4}
X_i =\frac{C_i}{C},
\end{equation}
where $C=\sum^N_{j=1} C_j$ is the total number of moles per unit volume.  As for \eqref{e3.3}, we also have
\begin{equation}\label{e3.5}
0\leq X_i\leq 1 \text{ for }   1 \leq i \leq N, \quad\sum^N_{i=1}X_i=1.
\end{equation}

Simple and useful relations between the $X_i$ and $Y_i$ are derived below.  At this point, we proceed with the definition of kinematical quantities.

The average velocity of molecules of type $i,$ at $x$ at time $t,$ is given by
\begin{equation*}
\boldsymbol{\bar{v}}_i(x,t) =\frac{1}{N_i}\int_{\mathbf{R}^n}\xi f_i(x,\xi, t)d\xi.
\end{equation*}
The mass-weighted average velocity of the mixture is
\begin{equation}\label{e3.6}
\boldsymbol{v} =\sum^N_{i=1}Y_i\boldsymbol{\bar{v}}_i ,
\end{equation}
which is the ordinary flow velocity considered in fluid dynamics.
\vskip0.1in

The relative velocity of species $i$ is given by
\begin{equation*}
\textbf{V}_i(x,t)=\boldsymbol{\bar{v}}_i(x,t) - \boldsymbol{v}(x,t),\quad i=1,\ldots ,N,
\end{equation*}
and, due to \eqref{e3.3} and \eqref{e3.6}, we have
\begin{equation}\label{e3.7}
\sum^N_{i=1}Y_i\textbf{V}_i = 0.
\end{equation}
\vskip0.1in

The Stefan-Maxwell equations express the gradients of the $X_i$ in terms of the $\boldsymbol{V}_i:$
\begin{equation}\label{e3.8}
\boldsymbol{\nabla} X_i = \sum^N_{j=1;j\not=i} d_{ij} X_iX_j(\textbf{V}_j-\textbf{V}_i), \quad i=1,\ldots , N,
\end{equation}
where $d_{ij}=\kappa/D_{ij} > 0,$ and $D_{ij} = D_{ji}, i\not= j$, is the binary diffusion coefficient for species $i$ and $j,$ while $\kappa$ is a constant representing thermal diffusion coefficients.

We are interested in the fluxes $\textbf{F}_i$ given by
\begin{equation}\label{e3.9}
\textbf{F}_i=Y_i\textbf{V}_i.
\end{equation}
In particular, we aim to show that the $\textbf{F}_i$ can be determined in term of the $Y_j$ and $\boldsymbol{\nabla} Y_j$ through equations \eqref{e3.7} and \eqref{e3.8}
and the $\boldsymbol{X}-\boldsymbol{Y}$ relations, $\boldsymbol{X} = (X_1,\ldots ,X_N),\; \boldsymbol{Y} = (Y_1,\ldots ,Y_N).$

We conclude this section by describing the relations between the $X_i$ and $Y_i$.  In view of \eqref{e3.1}, \eqref{e3.2} and \eqref{e3.4}, we obtain for
$i=1,\ldots ,N$
\begin{equation*}
Y_i=\frac{M_iC_i}{\sum\limits^N_{{j=1}}M_{j}C_{j}}, \quad X_i=\frac{C_i}{\sum\limits^N_{{j=1}}C_j},
\end{equation*}
and setting
\begin{equation}\label{e3.10}
Y_M=\sum^N_{j=1}\frac{Y_j}{M_j}, \qquad X_M=\sum^N_{j=1}M_jX_j,
\end{equation}
we have
\begin{equation}\label{e3.11}
Y_MX_M=1,
\end{equation}
\begin{equation}\label{e3.12}
Y_i=\frac{M_iX_i}{X_M}, \quad X_i =\frac{Y_i}{M_iY_M}.
\end{equation}
Also, setting
\begin{equation}\label{e3.13}
\underline{M}=\min_{1\leq i\leq N}M_i,\quad \overline{M}=\max_{1\leq i\leq N}M_i,\quad \widetilde{M}=\overline{M}/\underline{M},
\end{equation}
\eqref{e3.10} yields readily since $X_i \geq 0$, $Y_i \geq 0$, and $\sum^N_{i=1} Y_i=\sum^N_{i=1} X_i=1:$
\begin{equation}\label{e3.14}
\frac{1}{\overline{M}}\leq Y_M\leq\frac{1}{\underline{M}}, \quad{\underline{M}}\leq X_M\leq \overline{M}.
\end{equation}
\vskip0.1in
Next, we turn to some relations between the gradients of $X_i$ and $Y_i.$  We assume for the moment that the $X_i$ and $Y_i$ are smooth functions (say $C^1$) of
$(x,t)$  for $x\in\Omega$ and $t\in(0,T).$  In any event, the relations are pointwise relations, which are derived independently of the location $(x,t)$.  From \eqref{e3.12}, we have
\begin{equation}\label{e3.15}
\boldsymbol{\nabla} Y_i=\frac{M_i}{X_M}\boldsymbol{\nabla} X_i - \frac{M_iX_i}{X^2_M}\boldsymbol{\nabla} X_M,
\end{equation}
\begin{equation}\label{e3.16}
\boldsymbol{\nabla} X_i=\frac{1}{M_iY_M}\boldsymbol{\nabla} Y_i - \frac{Y_i}{M_i(Y_M)^2}\boldsymbol{\nabla} Y_M.
\end{equation}
Also, setting
\begin{equation*}
|\boldsymbol{\nabla X}|^2 = \sum^N_{i=1}|\boldsymbol{\nabla} X_i|^2,\quad|\boldsymbol{\nabla Y}|^2 = \sum^N_{i=1}|\boldsymbol{\nabla} Y_i|^2,
\end{equation*}
we obtain
\begin{equation}\label{e3.17}
\frac{1}{2N(\widetilde{M})^2} |\boldsymbol{\nabla X}|\leq |\boldsymbol{\nabla Y}|\leq 2N(\widetilde{M})^2|\boldsymbol{\nabla X}|.
\end{equation}
Indeed, with \eqref{e3.13}-\eqref{e3.15},
\begin{equation*}
|\boldsymbol{\nabla} Y_i|\leq\widetilde{M}|\boldsymbol{\nabla} X_i|+\widetilde{M}^2\sum^N_{j=1}|\boldsymbol{\nabla} X_j|\leq 2\widetilde{M}^2\sum^N_{j=1}|\boldsymbol{\nabla} X_j|,
\end{equation*}
which gives readily the second inequality in \eqref{e3.17}.  The proof of the first one is similar by making use of \eqref{e3.16}.

\subsection{The fluxes $\boldsymbol{F}_i$}\label{ss3.2}\hspace*{\fill} \\
From the physical context the $i^{th}$ species is absent in a region where $Y_i=0,$ so that $\textbf{V}_i$ does not make sense and $\textbf{F}_i=0$ in
such a region.  We will see how this is reflected in a purely algebraic study of the Stefan-Maxwell equations.

We start with some remarks concerning the linear equations (for the $\textbf{V}_i$) \eqref{e3.7}, \eqref{e3.8}.  By making use of \eqref{e3.12}, let us rewrite
them as
\begin{equation}\label{e3.18}
\sum^N_{i=1}Y_i\textbf{V}_i=0
\end{equation}
\begin{equation}\label{e3.19}
B(\boldsymbol{Y})\textbf{V}=\textbf{P},
\end{equation}
where
\begin{equation}\label{e3.20}
B_{ij}(\boldsymbol{Y})=\begin{cases}
&-d'_{ij} Y_iY_j\text{ for }j\not= i,\\
&\\
&\sum\limits^N_{k=1;k\not= i} d'_{ik}Y_iY_k\text{ for }j=i,
\end{cases}
\end{equation}
with
\begin{equation*}
d'_{ij}=\frac{d_{ij}}{M_iM_j},
\end{equation*}
and
\begin{equation*}
\textbf{P}=(\textbf{P}_1,\ldots ,\textbf{P}_N), \quad\textbf{P}_i = -Y^2_M\boldsymbol{\nabla} X_i,
\end{equation*}
is given by \eqref{e3.10} and \eqref{e3.16} in terms of $\boldsymbol{Y}$ and $\boldsymbol{\nabla Y}.$  Clearly, \eqref{e3.18}, \eqref{e3.19} is a system of $N+1$ vectorial equations with $N$ vectorial unknowns (vectors of $\mathbb{R}^n$).  Also, the matrix $B(\boldsymbol{Y})$ is symmetric and semi-definitive positive, since
\begin{equation}\label{e3.21}
\begin{split}
(B(\boldsymbol{Y})\textbf{V,V}) &= \sum^N_{i,j=1; i\not= j} d'_{ij} Y_iY_j (\textbf{V}_i-\textbf{V}_j) \cdot\textbf{V}_i,\\
&=\sum^N_{i,j=1; i<j} d'_{ij} Y_iY_j |\textbf{V}_i-\textbf{V}_j|^2, \text{ due to }d'_{ji} =d'_{ij}.
\end{split}
\end{equation}
It is worth mentioning also that
\begin{equation}\label{e3.21a}
\sum^N_{i=1} B_{ij}(\boldsymbol{Y}) = 0, \; j=1, \ldots , N;\qquad\sum^N_{j=1} B_{ij}(\boldsymbol{Y}) = 0, \; i=1, \ldots, N.
\end{equation}

\vskip0.1in
It follows from \eqref{e3.21} that if all the $Y_i$ are strictly positive, the matrix $B(\boldsymbol{Y})$ has rank $N-1.$  In that case, since $\sum^N_{i=1}\textbf{P}_i=-Y^2_M\sum^N_{i=1}\boldsymbol{\nabla} X_i = 0$ (a.e.) as $\sum^N_{i=1}X_i =1$, the equations \eqref{e3.19} are consistent, so that equations \eqref{e3.18} and \eqref{e3.19} determine uniquely the $\textbf{V}_i.$  In summary, from a strictly algebraic point of view, if all $Y_i$ are strictly positive and
\begin{equation}\label{e3.21a}
\sum^N_{i=1} \textbf{P}_i=0,
\end{equation}
equations \eqref{e3.18}, \eqref{e3.19} uniquely determine $\textbf{V}_1,\ldots ,\textbf{V}_N.$

\vskip0.1in
If, say, $Y_1,\ldots, Y_k$ are $>0$ and $Y_{k+1}=\ldots = Y_N=0,$ then, by inspection of the matrix $B(\boldsymbol{Y})$, we see, as before, that $\textbf{V}_1, \ldots, \textbf{V}_k$ are uniquely determined.  The remaining equations for $\textbf{V}_{k+1}, \ldots, \textbf{V}_N$ have no solutions unless $\textbf{P}_i=0$, $i=k+1, \ldots , N$, in which case the corresponding $\textbf{V}_i$ are arbitrary, and $\textbf{F}_i = Y_i\textbf{V}_i=0,$ for $i=k+1,\ldots , N.$\footnote{From the analytical point of view (by opposition to the algebraic point of view), in a region where $Y_i=X_i=0, \; \textbf{P}_i=- Y^2_M\boldsymbol{\nabla} X_i=0$ a.e.}  We come back below to the definitions of the $\textbf{F}_i.$  If some of the $Y_i$ vanish, it is not possible to determine uniquely the $\textbf{V}_i.$  However, we are only interested in defining the fluxes $\textbf{F}_i$ and we will show later on that this is indeed possible.

\vskip0.1in
Since $Y_i \geq 0$ and $\sum^N_{i=1}Y_i =1$ in the case of interest to us, we must continue to study the resolution of the linear system \eqref{e3.18}-\eqref{e3.19} in the case where $Y_i\geq 0 \; \forall i,$ while not all of the $Y_i$ vanish, and $\textbf{P}$ is a vector of $\mathbb{R}^{Nn},$ not necessarily equal to $-Y^2_M\boldsymbol{\nabla} \textbf{X}.$

\vskip0.1in
At this point, let us continue to study the case $Y_i>0,\enspace\forall i.$  The above argument for existence and uniqueness clearly breaks the symmetry with respect to the unknowns $\textbf{V}_1, \ldots\textbf{V}_N,$ one of the equations in \eqref{e3.19} being replaced by \eqref{e3.18}.  To avoid this difficulty, we aim to give a different formulation of \eqref{e3.18}-\eqref{e3.19}.  For that purpose, let us introduce the quantity
\begin{equation}\label{e3.22}
\begin{split}
(B(\boldsymbol{Y})\textbf{V}, \textbf{V}) &+ \gamma\left|\sum^N_{i=1}Y_i\textbf{V}_i\right|^2\\
& = (\text{with }\eqref{e3.21})\\
&=\sum^N_{i,j=1,i<j} d'_{ij}Y_iY_j |\textbf{V}_i-\textbf{V}_j|^2 +\gamma\left|\sum^N_{i=1}Y_i\textbf{V}_i\right|^2.
\end{split}
\end{equation}
Setting
\begin{equation*}
\underline{d}\ '=\min_{i,j}d'_{ij}, \quad\overline{d}'=\max_{i,j} d'_{ij},
\end{equation*}
and $\gamma = \underline{d}',$ we infer from \eqref{e3.22} that
\begin{equation}\label{e3.23}
\begin{split}
(B(\boldsymbol{Y})\textbf{V},\textbf{V}) + \gamma\left|\sum^N_{i=1}Y_i\textbf{V}_i\right|^2 &\geq\gamma
\left\{
\sum^N_{i,j=1,i<j} Y_iY_j(|\textbf{V}_i|^2 +|\textbf{V}_j|^2) +\sum^N_{i=1}Y^2_i|\textbf{V}_i|^2
\right\}\\
&\geq\gamma\sum^N_{i,j=1}Y_iY_j|\textbf{V}_i|^2 =\gamma \left(\sum^N_{j=1}Y_j\right)\left(\sum^N_{i=1}Y_i|\textbf{V}_i|^2\right).
\end{split}
\end{equation}
Therefore, the $N\times N$ matrix $C(\boldsymbol{Y})$ defined by
\begin{equation}\label{e3.25}
C_{ij}(\boldsymbol{Y}) = B_{ij}(\boldsymbol{Y}) +\gamma Y_iY_j,\quad 1\leq i,j\leq N,
\end{equation}
is positive definite when $Y_j>0,\enspace\forall j$ (and without any assumption on $\sum^N_{j=1}Y_j)$ since 
\begin{equation}\label{e3.25a}
\sum^N_{i,j=1} C_{ij}(\boldsymbol{Y})\textbf{V}_j\geq\gamma \left(\sum^N_{j=1}Y_j\right)\left(\sum^N_{i=1}Y_i|\textbf{V}_i|^2\right).
\end{equation}
In particular, when all $Y_j$ are strictly positive, the problem
\begin{equation}\label{e3.26}
C(\boldsymbol{Y})\textbf{V} =\textbf{P},
\end{equation}
where $\textbf{P} =(\textbf{P}_1, \ldots, \textbf{P}_N)$, $\textbf{P}_i\in\mathbb{R}^n$, has a unique solution.
\vskip0.1in

Note that this problem is equivalent to \eqref{e3.18}-\eqref{e3.19} when $\sum^N_{i=1} \textbf{P}_i=0$ and all $Y_j$ are strictly positive (even if $\sum^N_{j=1}Y_j\not=1$).  Indeed, if $\textbf{V}$ is the solution of \eqref{e3.18}-\eqref{e3.19}, then
\begin{equation*}
\sum^N_{j=1}C_{ij}(\boldsymbol{Y})\textbf{V}_j =\sum^N_{j=1}B_{ij} \ (\boldsymbol{Y})\textbf{V}_j +\gamma Y_i\sum^N_{j=1} Y_j\textbf{V}_j = \textbf{P}_i.
\end{equation*}
Conversely, if $C(\boldsymbol{Y})\textbf{V}=\textbf{P},$ then, on the one hand, by adding the equations, we find that
\begin{equation}\label{e3.27}
\sum^N_{i,j=1} C_{ij}(Y)\textbf{V}_j =\sum^N_{i=1}\textbf{P}_i=0,
\end{equation}
while, on the other hand, since $\sum^N_{i=1}B_{ij}(Y) = 0$, we have
\begin{equation*}
\sum^N_{i,j=1} C_{ij}(\boldsymbol{Y})\textbf{V}_j =\sum^N_{i,j=1} B_{ij}(\boldsymbol{Y})\textbf{V}_j +\gamma \left(\sum^N_{i=1}Y_i\right)\sum^N_{j=1} Y_j\textbf{V}_j =\gamma
\left(\sum^N_{i=1}Y_i\right)\sum^N_{j=1}Y_j\textbf{V}_j.
\end{equation*}
Combining these two equalities successively gives \eqref{e3.18} and \eqref{e3.19}.  

Now \eqref{e3.26} is an invertible system of $Nn$ equations for $Nn$ unknowns, which is symmetric with respect to the unknowns.

\vskip0.1in
We aim now to address the general case where some but not all $Y_i$ vanish.  As already mentioned, we can not define $\textbf{V}_i$ in general but, as we will see, we can define the $\textbf{F}_i$.  We assume again that the $\textbf{P}_i$ are arbitrary vectors of $\mathbb{R}^n$ $(\textbf{P}_i\not= - Y^2_M\boldsymbol{\nabla} X_i)$ such that
\begin{equation}\label{e3.27a}
\sum^N_{i=1}\textbf{P}_i=0,
\end{equation}
and, replacing $Y_i\textbf{V}_i$ by $\textbf{F}_i,$ we rewrite \eqref{e3.18}-\eqref{e3.19} in the form
\begin{equation}\label{e3.28}
\sum^N_{i=1}\textbf{F}_i=0,
\end{equation}
\begin{equation}\label{e3.29}
\left(\sum^N_{k=1;k\not= i}d'_{ik} Y_k\right)\textbf{F}_i
%- Y_i \sum^N_{j=1; j\not=i} \left(d'_{ik}Y_k\right)\textbf{F}_i
-Y_i\sum^N_{j=1; j\not= i} d'_{ij}\textbf{F}_j =\textbf{P}_i,\enspace 1\leq i\leq N.
\end{equation}
We consider again $\gamma = \underline{d}'$ as in \eqref{e3.23} and rewrite the linear system \eqref{e3.26}, replacing
$Y_i \textbf{V}_i$ by $\textbf{F}_i.$  We obtain (compare to \eqref{e3.29}):
\begin{equation}\label{e3.30}
\begin{split}
\left(\sum^N_{j=1;j\not= i}d'_{ij} Y_j+\gamma Y_i\right)\textbf{F}_i  & -Y_i\sum^N_{j=1; j\not= i}\left(d'_{ij} -\gamma\right)\textbf{F}_j\\
& = \textbf{P}_i, \quad 1\leq i\leq N.
\end{split}
\end{equation}
\vskip0.1in

As for \eqref{e3.26}, we show that, when \eqref{e3.27a} is satisfied, \eqref{e3.28}-\eqref{e3.29} is equivalent to \eqref{e3.30}.  Indeed, it is clear that \eqref{e3.28}-\eqref{e3.29} imply \eqref{e3.30}.  Conversely if the $\textbf{F}_i$ satisfy equations \eqref{e3.30} then, by adding these equations for $i=1, \ldots, N,$ we obtain
\begin{equation}\label{e3.31}
\gamma \left(\sum^N_{i=1}Y_i\right)\left(\sum^N_{j=1}\textbf{F}_j\right) =\sum^N_{j=1}\textbf{P}_j.
\end{equation}
Hence \eqref{e3.28} follows from \eqref{e3.27a}; then equations \eqref{e3.30} reduce to equations \eqref{e3.29}.

\vskip0.1in
We claim that \eqref{e3.30} possesses a unique solution, even if \eqref{e3.27a} is not satisfied.  Let us assume, say that $Y_1, \ldots, Y_k >0,$ while $Y_{k+1}= \ldots = Y_N=0;$ equations \eqref{e3.30} give for $i=k+1, \ldots N,$
\begin{equation}\label{e3.32}
\textbf{F}_i =\textbf{P}_i/S_i, \enspace \text{ with } \enspace S_i =S_i(\boldsymbol{Y}) = \sum^k_{j=1}d'_{ij}Y_j,\enspace  i=k+1, \ldots N.
\end{equation}
For $i=1, \ldots , k, $ the remaining system \eqref{e3.30} reads
\begin{equation}\label{e3.33}
\begin{split}
\left(\sum^k_{j=1; j\not= i}d'_{ij}Y_j +\gamma Y_i\right)\textbf{F}_i &- Y_i\sum^k_{j=1;j\not= 1} \left(d'_{ij} -\gamma\right)\textbf{F}_j\\
&= \textbf{P}_i + Y_i\sum^N_{j=k+1}\left( d'_{ij} -\gamma\right)\textbf{F}_j,
\end{split}
\end{equation}
where $\gamma = \underline{d}'$ again.  Writing $\textbf{F}_i= Y_i\textbf{V}_i,$ the system \eqref{e3.33} is similar to \eqref{e3.26} and it can be shown in the same way that it defines the $\textbf{F}_i$, $i=1, \ldots, k,$ uniquely.
\vskip0.1in

To summarize, we have shown that, for every $\textbf{P} = (\textbf{P}_1, \ldots \textbf{P}_N)\in\mathbb{R}^{Nn}$, \eqref{e3.30} possesses a unique solution $\textbf{F}=(\textbf{F}_1,\ldots , \textbf{F}_N),$ provided $Y_i\geq 0,\enspace\forall i,$ and not all the $Y_i$ vanish.  Furthermore, in view of \eqref{e3.31}, \eqref{e3.28} holds if and only if \eqref{e3.27a} is assumed, and in this case \eqref{e3.30} is equivalent to \eqref{e3.28}-\eqref{e3.29}.
\vskip0.1in

We summarize this study in the following theorem.
\begin{thm}\label{t3.1}
Let $\textbf{P}_i, 1 \leq i \leq N$, be arbitrary vectors of $\mathbb{R}^n$ satisfying the physically relevant condition:
\begin{equation*}
\sum^N_{i=1}\textbf{P}_i=0.
\end{equation*}
We consider the Stefan-Maxwell equations rewritten in the form \eqref{e3.18}, \eqref{e3.19} for the $\textbf{V}_i,$ or in the form \eqref{e3.28}, \eqref{e3.29} for the $\textbf{F}_i$, where $(Y_1,,\ldots ,Y_N) \in \mathbb{R}^N$ is given, with $Y_i\geq 0 \; \forall i$ and not all of the $Y_i$ vanish.
\vskip0.1in
\noindent (i) If $Y_i > 0 \;\forall i$, these $N+1$ linear equations are consistent and define the $\textbf{V}_i$ and $\textbf{F}_i=Y_i\textbf{V}_i$ uniquely. Furthermore, the $\textbf{V}_i$ are the solutions of 
the linear system \eqref{e3.26} which has a symmetric positive matrix.

\vskip0.1in
\noindent (ii)  If some of the $Y_i$ are zero but not all of them, say if $Y_1, \ldots , Y_k>0, Y_{k+1} = \ldots = Y_N=0, \textbf{V}_1, \ldots ,\textbf{V}_k$ are uniquely defined and $\textbf{V}_{k+1}, \ldots, \textbf{V}_N$ are undetermined.  In this case all the $\textbf{F}_i$ are uniquely determined and are given by \eqref{e3.32} and the resolution of the linear system \eqref{e3.33} of order $k$. Furthermore $\textbf{F}_{k+1} = \ldots =\textbf{F}_N =0$ in the (relevant) case where $\textbf{P}_{k+1} = \ldots =\textbf{P}_N=0.$

\vskip0.1in
\noindent (iii) In all cases, the $\textbf{F}_i$ are uniquely determined and solutions of the linear system \eqref{e3.30} which has an invertible matrix.
\end{thm}

\subsection{More about the fluxes}\label{ss3.3}\hspace*{\fill} \\
We want now to derive some properties of the fluxes $\textbf{F}_i$ that are the solutions to the linear system \eqref{e3.30}.  Obviously using Cramer's rule, we can write
\begin{equation}\label{e3.34}
\textbf{F}_i=\sum^N_{j=1} f_{ij}(Y_1, \ldots ,Y_N)\textbf{P}_j,
\end{equation}
where the $f_{ij}$ are rational functions with respect to the $Y_{j},$ defined on 
$\mathbb{R}^N_+\backslash\left\{(0,\ldots ,0)\right\}$ where $\mathbb{R}^N_+ = [0,+\infty)^N$.  Also, comparing
\eqref{e3.32} and \eqref{e3.34} we see that, for $i\not= j,\; f_{ij}$ vanishes at $Y_i=0,$ so that
\begin{equation}\label{e3.35}
\begin{split}
&f_{ij}(Y_1,\ldots ,Y_N) =Y_i\tilde{f}_{ij} (Y_1, \ldots ,Y_N), \text{ where }\tilde{f}_{ij}\text{ is }\\
&\text{a rational function continuous on }\mathbb{R}^N_+\backslash\left\{(0,\ldots, 0)\right\}.
\end{split}
\end{equation}
\vskip0.1in

Recall that if $Y_i>0,\enspace\forall i,$ then $\textbf{F}_i =Y_i\textbf{V}_i$ and the $\textbf{V}_i$ are solutions of \eqref{e3.26}.  Since the matrix $C(\boldsymbol{Y})$ is definite positive,  the inversion of \eqref{e3.26} gives $\textbf{V}_i=\sum^N_{j=1} D_{ij}(\boldsymbol{Y})\textbf{P}_j$ where $D(\boldsymbol{Y}) = C(\boldsymbol{Y})^{-1}$ is symmetric definite positive.  In particular, since $D_{ii}(\boldsymbol{Y})\geq 0$, the decomposition \eqref{e3.34} of $\textbf{F}_i = Y_i\textbf{V}_i$ is such that
\begin{equation}\label{e3.36}
f_{ii}(\boldsymbol{Y})\geq 0\text{ on } ]0,+\infty)^N\text{ and, by continuity, on }\mathbb{R}^N_+\backslash\left\{(0,\ldots , 0)\right\}.
\end{equation}
\vskip0.1in

Let us specialize this result to the case where the $Y_i$ are functions from $\Omega$ into $\mathbb{R}_+$, say of class $C^1$, such that $\sum^N_{i=1}Y_i(x)\not=0$ at each point
$x\in\Omega$ and $\textbf{P}_i=\textbf{P}_i(x)=-Y^2_M\boldsymbol{\nabla} X_i(x)$ where
\begin{equation}\label{e3.37}
X_i=\frac{Y_i}{M_iY_M}, \quad Y_M =\sum^N_{j=1}\frac{Y_j}{M_j}>0.
\end{equation}
Then \eqref{e3.34} becomes
\begin{equation}\label{e3.38}
\textbf{F}_i=-\sum^N_{j=1} f_{ij}(Y_1, \ldots, Y_N)Y^2_M\boldsymbol{\nabla} X_j.
\end{equation}
Since $\sum^N_{i=1} X_i=1$\footnote{By \eqref{e3.37}, $\sum^N_{i=1}X_i=1$, is valid although $\sum^N_{i=1}Y_i$ may not be equal to one.}, \eqref{e3.27a} is satisfied so that \eqref{e3.28} is satisfied too and  reads
\begin{equation}\label{e3.39}
\sum^N_{i,j=1}f_{ij}(Y_1, \ldots, Y_N)\boldsymbol{\nabla} X_j=0.
\end{equation}
Then we express the $\boldsymbol{\nabla} X_j$ in terms of the $\boldsymbol{\nabla} Y_\ell:$
\begin{equation*}
\boldsymbol{\nabla} X_j = \frac{\boldsymbol{\nabla} Y_j}{M_jY_M} -\frac{Y_j}{M_jY^2_M}\sum^N_{\ell = 1}\frac{\boldsymbol{\nabla} Y_\ell}{M_\ell},
\end{equation*}
and the fluxes $\textbf{F}_i$ become
\begin{equation}\label{e3.40}
\textbf{F}_i =-\sum^N_{j=1}\tilde a_{ij}(Y_1, \ldots Y_N)\boldsymbol{\nabla} Y_j,
\end{equation}
\begin{equation}\label{e3.41}
\tilde a_{ij}=\frac{f_{ij}Y_M}{M_j} - \sum^N_{\ell =1}\frac{Y_\ell f_{i\ell}}{M_jM_\ell }.
\end{equation}
Therefore, we infer from the properties of the $f_{ij}$ that
\begin{equation}\label{e3.42}
\tilde a_{ij}(Y_1, \ldots, Y_N) =
\begin{cases}
&Y_ia^*_{ij}(Y_1, \ldots, Y_N),\text{ if }i\not= j,\\
&b^{*0}_i(Y_1, \ldots ,Y_N) + Y_ib^{*1}_i(Y_1, \ldots, Y_N), \text{ if }i=j,
\end{cases}
\end{equation}
where
\begin{equation}\label{e3.43}
\begin{cases}
&a^*_{ij}, b^{*0}_i\text{ and }b^{*1}_i\text{ are rational functions of }Y_1, \ldots, Y_N,\\
&\text{continuous in }\mathbb{R}^N_+\backslash\left\{(0, \ldots, 0)\right\}\text{ and }b^{*0}_i\geq 0.
\end{cases}
\end{equation}
Also since \eqref{e3.40} is just a rewriting of \eqref{e3.38}, \eqref{e3.39} implies
\begin{equation}\label{e3.44}
\sum^N_{i,j=1}\tilde a_{ij}(Y_1, \ldots Y_N)\boldsymbol{\nabla} Y_j=0,
\end{equation}
provided, as before, that the $Y_i$ are functions from $\Omega$ into $\mathbb{R}_+$ such that $\sum^N_{i=1}Y_i(x)\not= 0$ at each point
$x\in\Omega$; in particular $\sum^N_{j=1}\boldsymbol{\nabla} Y_j(x)=0$ is not required for \eqref{e3.44}.

\vskip0.1in
At this point we have shown that the Stefan-Maxwell equations allow us to define the fluxes $\textbf{F}_i$ (but not necessarily the $\textbf{V}_i$) provided that
$Y_j\geq 0,\forall j,$ and not all $Y_j$ vanish.  For the mathematical study (see  \eqref{e2.8}-\eqref{e2.11}) we will need the fluxes to be defined for $\boldsymbol{Y}=(Y_1,\ldots ,Y_N)=(0,\ldots , 0)$. Clearly if
all $Y_i$ vanish, equations \eqref{e3.18}-\eqref{e3.19}
are not valid since in \eqref{e3.8} we can not express $X_i$ in terms of the $Y_j$ by \eqref{e3.12} ($Y_M$ vanishes). This leads us to introduce modified expressions of the fluxes, that is modifications of the coefficients in \eqref{e3.34} and \eqref{e3.40}.  The new coefficients will be defined on all of $\mathbb{R}^N_+$ and the corresponding new fluxes will coincide with the previous ones, provided that
\begin{equation}\label{e3.45}
\sum^N_{j=1}Y_j=1.
\end{equation}
Since we will be able to show that the solution of \eqref{e2.4} (supplemented with the boundary and initial conditions) satisfies \eqref{e3.45}, the fluxes in \eqref{e2.4} will indeed be the ones given by the Stefan-Maxwell equations, so that our modification is licit.
\vskip0.1in

Coming back to the expression \eqref{e3.34} for the fluxes, the coefficients $f_{ij}(Y_1, \ldots ,Y_N)$ are rational functions continuous on $\mathbb{R}^N_+\backslash\left\{(0, \ldots, 0)\right\}$, but with a singularity at $(0, \ldots, 0)$. In order to preserve convenient properties of these coefficients, it is useful to rewrite the corresponding fractions with the same positive denominator
\begin{equation*}
f_{ij}=\frac{g_{ij}}{h},
\end{equation*}
where the $g_{ij}$ and $h$ are polynomial functions of $\boldsymbol{Y}$, h does not vanish in $\mathbb{R}^N_+\backslash\left\{(0, \ldots, 0)\right\}$ and is positive
\footnote{That is $a_1/b_1,\ldots ,a_N/b_N,$ are written as fractions with denominator $(b_1,\ldots b_N)^2.$ }. Setting
\begin{equation*}
\overline{f}_{ij} = \frac{h}{h+\left(\sum^N_{\ell=1} Y_\ell-1\right)^2} f_{ij}=
\frac{g_{ij}}{h+\left(\sum^N_{\ell=1} Y_\ell-1\right)^2},
\end{equation*}
the new coefficients are rational functions of $Y_1, \ldots, Y_N,$ defined and continuous on all of $\mathbb{R}^N_+$ which coincide with $f_{ij}$ if
$\sum^N_{j=1}Y_j=1$ and they satisfy properties analogous to \eqref{e3.35} and \eqref{e3.36}.  Also if we set
\begin{equation*}
\overline{\textbf{F}}_i =\sum^N_{j=1}\overline{f}_{ij}(Y)\textbf{P}_j,
\end{equation*}
we still have
\begin{equation}\label{e3.46}
\sum^N_{i=1}\overline{\textbf{F}}_i=0\enspace\text{ if }\enspace\sum^N_{i=1}\textbf{P}_i=0.
\end{equation}
\vskip0.1in
Next we replace $\tilde{a}_{ij}$ in \eqref{e3.41} by
\begin{equation}\label{e3.47}
a_{ij} = \frac{h}{h+\left(\sum^N_{\ell=1} Y_\ell-1\right)^2} \tilde{a}_{ij}= \frac{\bar{f}_{ij}Y_M}{M_j} - \sum^N_{\ell =1}\frac{Y_\ell\bar{f}_{i\ell}}{M_jM_\ell}.
\end{equation}
The $a_{ij}$ are rational functions of the $Y_\ell,$ continuous on all of $\mathbb{R}^N_+,$ taking the same values as $\tilde{a}_{ij}$ if
$\sum^N_{j=1}Y_j=1.$  They satisfy properties analogous to \eqref{e3.42} and \eqref{e3.43}.  Furthermore, setting
\begin{equation}\label{e3.48b}
\tilde{\textbf{F}}_i = -\sum^N_{j=1} a_{ij}(Y_1, \ldots , Y_N)\boldsymbol{\nabla} Y_j,
\end{equation}
we have $\tilde{\textbf{F}}_i=\textbf{F}_i$ when $\sum^N_{j=1}Y_j=1,$ and
\begin{equation*}
\sum^N_{i=1}\tilde{\textbf{F}}_i=0\enspace\text{ since }\enspace\sum^N_{i=1}\textbf{F}_i=0,
\end{equation*}
i.e.
\begin{equation}\label{e3.48}
\sum^N_{i,j=1}a_{ij}(Y_1,\ldots ,Y_N)\boldsymbol{\nabla} Y_j=0,
\end{equation}
even if $\sum^N_{j=1}\boldsymbol{\nabla} Y_j(x)$ does not vanish.
\vskip0.1in
Relations \eqref{e3.47} define the $a_{ij}$ in \eqref{e2.8}.  The smoothness assumptions as well as \eqref{e2.9}, \eqref{e2.10} and \eqref{e2.11} are satisfied.  The fluxes \eqref{e2.8} coincide with the ones given by the Stefan-Maxwell equations, provided that $\sum^N_{i=1}Y_i=1$.

In summary we have proven the following:
 \begin{thm}\label{t3.2}
For $i=1,...,N$, let $Y_i \in C^1(\Omega)$ (resp. $Y_i \in H^1(\Omega)$) be given such that $Y_i(x) \geq 0$ and $\sum^N_{i=1}Y_i(x)\not=0$ at each point
$x\in\Omega$ (resp. for a.e. $x\in\Omega$). Then the generalized fluxes $\tilde{\textbf{F}}_i(x)$ are given by \eqref{e3.48b} where the $a_{ij}$ are rational functions of the $Y_i$ defined and continuous on all of $\mathbb{R}^N_+$. 

Furthermore if $\sum^N_{i=1} Y_i(x)=1$ for $x \in \Omega$ (resp. for a.e. $x\in\Omega$), they coincide with the solutions of the linear system \eqref{e3.28}, \eqref{e3.29} with $\textbf{P}_i = -Y^2_M \boldsymbol{\nabla} X_i$ (and the Stefan-Maxwell equations). 

Finally the $a_{ij}$ satisfy the properties \eqref{e2.9}, \eqref{e2.10} and \eqref{e2.11}.
\end{thm}

The generalized fluxes $\tilde{\textbf{F}}_i$ are the ones we consider in equation \eqref{e2.4} and we will now denote them by $\textbf{F}_i$  for the sake of simplicity. However note that they coincide with the solutions of the Stefan-Maxwell equations given by Theorem 3.1 only if $\sum^N_{i=1} Y_i=1$.

There remains to derive the property \eqref{e2.12}.

\subsection{The property \eqref{e2.12}}\label{ss3.4}\hspace*{\fill} \\
We are now given $N$ functions $Y_1,\ldots , Y_N$ belonging to $H^1(\Omega)$ such that $0\leq Y_i(x)\leq 1$ and $\sum^N_{j=1} Y_j(x) = 1$ for a.e. $x\in\Omega$. 

Let us first assume that $x\in\Omega$ is such that $Y_i(x)>0 \; \forall i$ and $\sum^N_{i=1}Y_i(x)=1$.  Then, in view of Theorem 3.1 (i), the fluxes $\textbf{F}_i$ read $\textbf{F}_i=Y_i\textbf{V}_i,$ where $\textbf{V}_i$ is the solution of \eqref{e3.26} with $\textbf{P} = -Y^2_M \boldsymbol{\nabla} \boldsymbol{X}$, that is
\begin{equation}\label{e3.48d}
C(\boldsymbol{Y})\textbf{V} =-Y^2_M\boldsymbol{\nabla} \boldsymbol{X}.
\end{equation}
  Recalling the definition \eqref{e2.13} for $\mu_i$, we have
\begin{equation}\label{e3.48a}
\begin{split}
-\sum^N_{i=1}\textbf{F}_i\cdot\boldsymbol{\nabla}\mu_i &=-\sum^N_{i=1}\frac{Y_i}{M_iX_i}\boldsymbol{V}_i\cdot \boldsymbol{\nabla} X_i,\\
&=-Y_M\sum^N_{i=1}\boldsymbol{V}_i\cdot\boldsymbol{\nabla} X_i,\text{ with \eqref{e3.12}},\\
&= \frac{1}{Y_M}\sum^N_{i,j=1} C_{ij}(\boldsymbol{Y})\textbf{V}_j\cdot\textbf{V}_i,\text{ thanks to \eqref{e3.48d}}.
\end{split}
\end{equation}
Since $\sum^N_{j=1}Y_j(x)=1,$ the coercivity property \eqref{e3.25a} for the matrix $C(\boldsymbol{Y})$ reads 
$$\sum^N_{i,j=1} C_{ij}(\boldsymbol{Y})\textbf{V}_j\geq\gamma \left(\sum^N_{i=1}Y_i|\textbf{V}_i|^2\right),\enspace \text{ with }\gamma = \underline{d}',$$
and the bounds \eqref{e3.14} for $Y_M$ hold true. Therefore we infer from   \eqref{e3.48a}  that
\begin{equation}\label{e3.49}
-\sum^N_{i=1}\textbf{F}_i\cdot\boldsymbol{\nabla}\mu_i\geq\gamma \underline{M}\sum^N_{i=1}Y_i|\boldsymbol{V}_i|^2.
\end{equation}
%\vskip0.1in
Next, recall that $\textbf{V}_i$ is also the solution of  \eqref{e3.19} with $\textbf{P}_i =- Y^2_M\boldsymbol{\nabla} X_i$ so :
\begin{equation*}
-Y^2_M\boldsymbol{\nabla} X_i=\sum^N_{j=1}B_{ij}(\boldsymbol{Y})\boldsymbol{V}_j,
\end{equation*}
In view of the definition \eqref{e3.20} of $B_{ij}$ and using again \eqref{e3.14}, we find
\begin{equation}\label{e3.48c}
\begin{split}
|\boldsymbol{\nabla} X_i|&\leq \overline{M}\ ^2\sum^N_{j=1} |B_{ij}(\boldsymbol{Y})||\boldsymbol{V}_j|,\\
&\leq\overline{d}'\;\overline{M}\ ^2
\left\{
\sum^N_{k=1;k\not= i}Y_iY_k|\boldsymbol{V}_i|+\sum^N_{j=1;j\not= i}Y_iY_j|\boldsymbol{V}_j|
\right\},\\
&\leq\overline{d}' \;\overline{M}^2
Y_i^{1/2} \left(\sum^N_{k=1}Y_k|\boldsymbol{V}_k|^2\right)^{1/2} \left\{
\left(\sum^N_{k=1;k\not= i}Y_k\right)^2 + \sum^N_{j=1;j\not= i}Y_iY_j
\right\}^{1/2},\\
&\leq \overline{d}'\;\overline{M}^2
Y_i^{1/2}(1-Y_i)^{1/2}\left(\sum^N_{k=1}Y_k|\boldsymbol{V}_k|^2\right)^{1/2}.
\end{split}
\end{equation}
Thanks to the relations \eqref{e3.17} between $\boldsymbol{\nabla}\boldsymbol{X}$ and $\boldsymbol{\nabla}\boldsymbol{Y}$, we infer from \eqref{e3.48c} that
\begin{equation}\label{e3.50}
|\boldsymbol{\nabla}\boldsymbol{Y}|^2 = \sum^N_{i=1}|\boldsymbol{\nabla} Y_i|^2\leq c^2\left(\sum^N_{i=1}Y_i|\boldsymbol{V}_i|^2\right),
\end{equation}
where $c$ is an appropriate constant depending on $N, \overline{d}', \widetilde{M}, \overline{M}$. 

We conclude by combining \eqref{e3.49} and \eqref{e3.50}.  This provides
\begin{equation*}
-\sum^N_{i=1}\textbf{F}_i\cdot\boldsymbol{\nabla}\mu_i\geq c_1|\boldsymbol{\nabla}\boldsymbol{Y}|^2.
\end{equation*}
\vskip0.1in

Now, assume that $Y_{k+1}= \ldots = Y_N=0$ and $Y_1,\ldots ,Y_k>0$ at some $x\in\Omega$ and for some $k\geq 1.$  Then, for a.e. such $x\in\Omega,$ since $Y_i\in H^1(\Omega),$ we have
\begin{equation*}
\boldsymbol{\nabla} Y_i(x)=0,\enspace\boldsymbol{\nabla} X_i(x)=0,\enspace i=k+1,\ldots,N
\end{equation*}
so that $\textbf{P}_{k+1} = \ldots =\textbf{P}_N =0$. Therefore, in view of Theorem 3.1 (ii), $\textbf{F}_{k+1} = \ldots =\textbf{F}_N =0$ while $\textbf{F}_{1}, \ldots ,\textbf{F}_k$ are the solutions of \begin{equation*}
\left(\sum^k_{j=1;j\not= i}d'_{ij}Y_j +\gamma Y_i\right) \textbf{F}_i-Y_i\sum^k_{j=1;j\not= i}\left(d'_{ij} -\gamma\right)\textbf{F}_j=-Y^2_M\boldsymbol{\nabla} X_i, \enspace i=1,\ldots,k.
\end{equation*}
As already noticed, this system is similar to the previous one when all $Y_i$ are positive.  Therefore, with computations similar to the ones above, we find
\begin{equation*}
-\sum^N_{i=1}\textbf{F}_i\cdot\boldsymbol{\nabla}\mu_i\mathbbm{1}_{\left\{Y_i>0\right\}}= - \sum^k_{i=1} \textbf{F}_i\cdot\boldsymbol{\nabla}\mu_i \geq c_1\left(\sum^k_{j=1}|\boldsymbol{\nabla} Y_j|^2\right) = c_1\left(\sum^N_{j=1}|\boldsymbol{\nabla} Y_j|^2\right).
\end{equation*}
The above inequalities valid for various values of $k$ provide \eqref{e2.12}.

\begin{rem}\label{r3.1}
Note that our method of solutions of the Stefan-Maxwell equations reducing first the problem to the inversion of a symmetric positive definite matrix (the matrix  $C(\boldsymbol{Y})$ in \eqref{e3.25}-\eqref{e3.26}) is closely related to the one in  \cite{Gio90}, \cite{Gio91}. However our presentation above, contains some additional developments that are new to the best of our knowledge, in particular the generalized definition of the fluxes when all the $Y_i$ vanish, and the properties of these generalized fluxes, including the property (2.12).
\end{rem}

\subsection{The three species case}\label{ss3.5}\hspace*{\fill} \\
We conclude this Section \ref{s3} by studying explicitly the three species case $(N=3)$ which is of interest, as it includes for instance the evolution of ozone when the three species are atomic oxygen, molecular oxygen and ozone $(O, O_2, O_3$ respectively, see Appendix B in \cite{MMT93}).

The matrix $B(\boldsymbol{Y})$ in \eqref{e3.20} is written
\begin{equation}\label{e3.51}
B=
\left ( \begin{array}{ccc}
b+c & -c & -b \\
-c & a+c & -a \\
-b & -a & a+b \end{array} \right),
\end{equation}
where 
\begin{equation}\label{e3.51a}
a=d'_{23}Y_2Y_3, \enspace b=d'_{13}Y_1Y_3,\enspace c=d'_{12}Y_1Y_2.
\end{equation}
\vskip0.1in

The resolution of \eqref{e3.18}-\eqref{e3.19} (or more precisely of \eqref{e3.28}-\eqref{e3.29}) is much simplified by observing that
\begin{equation*}
DB=\rho I-
\left(\begin{array}{ccc}
bc& ac & ab \\
bc & ac & ab \\
bc & ac & ab\end{array}\right),
\end{equation*}
where $D$ is the diagonal matrix $(a,b,c)$ and $\rho = ab+bc+ca$, hence in view of \eqref{e3.51a}:
\begin{equation}\label{e3.52}
\rho = Y_1Y_2Y_3\tilde{\rho},\enspace \text{ with } \tilde{\rho} = d'_{13}d'_{23}Y_3 + d'_{12}d'_{13}Y_1 +d'_{12}d'_{23}Y_2.
\end{equation}
Here, when the $Y_i$ are positive and at least one of them does not vanish, we have $\tilde{\rho} >0$.

Now, multiplying both sides of equation \eqref{e3.19} by $D$, we find
\begin{equation}\label{e3.54}
\rho\boldsymbol{V} - (\boldsymbol{\sigma},\boldsymbol{\sigma},\boldsymbol{\sigma})^T = (a\boldsymbol{P}_1, b\boldsymbol{P}_2, c\boldsymbol{P}_3)^T,
\end{equation}
where
$\enspace\boldsymbol{\sigma} = bc\boldsymbol{V}_1 +ac\boldsymbol{V}_2+ab\boldsymbol{V}_3.
$
Taking the scalar product of \eqref{e3.54} with $\boldsymbol{Y}$ and using $\sum^3_{i=1}Y_i\boldsymbol{V}_i = 0$, we find
\begin{equation*}
\boldsymbol\sigma = -\left( \sum^3_{i=1}Y_i \right)^{-1}\quad (aY_1\boldsymbol{P}_1+bY_2\boldsymbol{P}_2+cY_3\boldsymbol{P}_3),
\end{equation*}
so that
\begin{equation}\label{e3.55}
\boldsymbol{V} = \frac{1}{\rho}
\left(\begin{matrix}
\boldsymbol{\sigma} +a\boldsymbol{P}_1\\
\boldsymbol{\sigma} +b\boldsymbol{P}_2\\
\boldsymbol{\sigma} +c\boldsymbol{P}_3
\end{matrix}\right).
\end{equation}
We recover that $\boldsymbol{V}$ is only defined when all $Y_i$ are strictly positive. However, recalling that $\boldsymbol{F}_i=Y_i\boldsymbol{V}_i$, we have:
\begin{equation}\label{e3.55}
\left(\begin{matrix}
\boldsymbol{F}_1\\
\boldsymbol{F}_2\\
\boldsymbol{F}_3
\end{matrix}\right) =-\frac{1}{\tilde{\rho}}\frac{1}{\sum^3_{i=1}Y_i}
\left(\begin{matrix}
d'_{23}Y_1\boldsymbol{P}_1 + d'_{13}Y_1\boldsymbol{P}_2+d'_{12}Y_1\boldsymbol{P}_3\\
d'_{23}Y_2\boldsymbol{P}_1 + d'_{13}Y_2\boldsymbol{P}_2+d'_{12}Y_2\boldsymbol{P}_3\\
d'_{23}Y_3\boldsymbol{P}_1 + d'_{13}Y_3\boldsymbol{P}_2+d'_{12}Y_3\boldsymbol{P}_3
\end{matrix}\right)
+
\frac{1}{\tilde{\rho}}
\left(\begin{matrix}
d'_{23}\boldsymbol{P}_1\\
d'_{13}\boldsymbol{P}_2\\
d'_{12}\boldsymbol{P}_3
\end{matrix}\right),
\end{equation}
which is defined when $Y_i \geq 0$ and not all of the $Y_i$ vanish. This gives the explicit form of the coefficients $f_{ij}$ in \eqref{e3.34}. We recover that they are rational functions defined and continuous on $\mathbb{R}^N_+\backslash\left\{(0,\ldots ,0)\right\}$ and that they satisfy the properties \eqref{e3.35} and \eqref{e3.36}.  

Setting $\textbf{P}_i = -Y^2_M \boldsymbol{\nabla} X_i$, and expressing the $\boldsymbol{\nabla} X_i$ in terms of the $\boldsymbol{\nabla} Y_j$ provide the coefficients in \eqref{e3.40} ; the properties  \eqref{e3.42} and \eqref{e3.43} follow as well.

\vskip0.1in

\section{The Chemistry System}\label{s4}

The aim of this section is to study the problem \eqref{e2.4} and to prove Theorem \ref{t2.1}.  For that purpose, we first introduce a modified problem depending on a parameter $\varepsilon >0$ for which we obtain an existence result.  Then we derive the existence of a solution of \eqref{e2.4} by taking the limit $\varepsilon\rightarrow 0$.

\subsection{The modified equations}\label{ss4.1}\hspace*{\fill} \\
We first modify and extend the coefficients $a_{ij}$ to be defined on $\mathbb{R}^N$ by setting
\begin{equation}\label{e4.1}
\hat{a}_{ij}(Y_1,\ldots Y_N) = \xi\left(\sum^N_{\ell =1}|Y_\ell |\right) a_{ij}(Y^+_1,\ldots , Y^+_N),\quad 1\leq i,j\leq N, \quad Y_k\in\mathbb{R},
\end{equation}
where $\xi :\mathbbm{R}_+\rightarrow\mathbbm{R}_+$ is a continuous function such that
$$\xi(s) =
s\text{ if }0\leq s\leq 1,\enspace \xi(s) \in [0,1] \text{ if }1\leq s\leq 2,\enspace
\xi(s) = 0\text{ if }s\geq 2.$$
Clearly, the $\hat{a}_{ij}$ are continuous bounded functions. The same is true for the $\omega_i$  given by \eqref{e4.2} and we set
\begin{equation}\label{e4.3}
K_1 = \max_{i,j}\sup_{\mathbb{R}^N} |\hat{a}_{ij}|,\quad K_2 = \max_i\sup_{\mathbb{R}^{N+1}}|\omega_i|.
\end{equation}

For $q>2$ fixed and for $\varepsilon >0$ fixed (which we will let converge to zero later on), we consider the following modified form of \eqref{e2.4}:
\begin{equation}\label{e4.4}
\begin{split}
\frac{\partial Y_i}{\partial t} + (\boldsymbol{v}\cdot\boldsymbol{\nabla}) Y_i & + \boldsymbol{\nabla}\cdot\boldsymbol{\hat{F}}_i\\
& - \varepsilon\boldsymbol{\nabla} \cdot (|\boldsymbol{\nabla}\boldsymbol{Y}|^{q-2}\boldsymbol{\nabla} Y_i) = \omega_i(\theta, Y_1,\ldots Y_N), \enspace 1\leq i\leq N.
\end{split}
\end{equation}
Here, $Y_i=Y_{i,\varepsilon}$ depends of course on $\varepsilon,$ but we omit to denote this dependence as long as $\varepsilon$ is kept fixed.  Also, $|\boldsymbol{\nabla} \boldsymbol{Y}|^2 = \sum^N_{j=1} |\boldsymbol{\nabla} Y_j|^2$ and
  \begin{equation}\label{e4.4b}
  \hat{\boldsymbol{F}}_i=-\sum^N_{j=1}\hat{a}_{ij} (Y_1,\ldots ,Y_N)\boldsymbol{\nabla}Y_j,
  \end{equation}
 where $\hat{a}_{ij}$ is given by \eqref{e4.1}. We supplement \eqref{e4.4} with the same boundary and initial conditions as before, namely \eqref{e2.22} and \eqref{e2.25} except that \eqref{e2.22}$_2$ is replaced by
\begin{equation}\label{e4.4a}
\boldsymbol{\nu}\cdot (\hat{\boldsymbol{F}}_i - \varepsilon |\boldsymbol{\nabla}\boldsymbol{Y}|^{q-2}\boldsymbol{\nabla}Y_i)=0\text{ on }\Gamma_h\cup\Gamma_\ell.
\end{equation}
To obtain the weak formulation of this problem we observe that $Y_i-Y^u_i$  vanishes at $x_n=0.$  Hence, upon multiplying \eqref{e4.4} by a smooth test function $z_i$ vanishing at $x_n=0,$ we obtain thanks to \eqref{e4.4a}:
\begin{equation}\label{e4.5}
\begin{split}
\int_\Omega\frac{\partial Y_i}{\partial t}&z_idx +\int_\Omega [(\boldsymbol{v}\cdot\boldsymbol{\nabla} )Y_i]z_idx+\sum^N_{j=1}\int_\Omega\hat{a}_{ij}(Y_1,\ldots ,Y_N)
\boldsymbol{\nabla} Y_j\cdot\boldsymbol{\nabla} z_idx\\
&+\varepsilon\int_\Omega |\boldsymbol{\nabla}\boldsymbol{Y}|^{q-2}\boldsymbol{\nabla} Y_i\cdot\boldsymbol{\nabla} z_idx = \int_\Omega \omega_i(\theta, Y_1,\ldots, Y_N)z_idx, \enspace 1\leq i\leq N.
\end{split}
\end{equation}
Let us introduce the Sobolev space $W^{1,q}(\Omega)$ and its subspace
\begin{equation*}
W^{1,q}_{\Gamma_{0}} (\Omega) = \left\{ z\in W^{1,q}(\Omega), \enspace z = 0\text{ at }x_n=0\right\}.
\end{equation*}
We denote by $(W^{1,q}_{\Gamma_{0}}(\Omega))'$ its dual, and $<\cdot ,\cdot >$ denotes the duality product between $W^{1,q}_{\Gamma_{0}}(\Omega)$ and its dual. Subsequently we replace in \eqref{e4.5}
\begin{equation*}
\int_\Omega\frac{\partial Y_i}{\partial t} z_idx\quad\text{ by }\quad\bigg<\frac{\partial Y_i}{\partial t}, z_i\bigg> .
\end{equation*}
\vskip0.1in

We now aim to prove the following existence result.
\begin{prop}\label{p3.1}
Under the assumptions of Theorem \ref{t2.1}, for $q>2$ and $\varepsilon >0$ given, problem \eqref{e4.4}, \eqref{e4.4a}, \eqref{e2.22}$_1$, \eqref{e2.25} possesses a solution
$\boldsymbol{Y}= (Y_1,\ldots ,Y_N)$ such that
\begin{equation}\label{e4.6}
Y_i\in L^\infty(0,T; L^2(\Omega))\cap L^q (0,T; W^{1,q}(\Omega)),
\end{equation}
\begin{equation}\label{e4.7}
\frac{\partial Y_i}{\partial t}\in L^{q{'}}(0,T; (W^{1,q}_{\Gamma_{0}}(\Omega))'), \enspace\text{ with }\frac{1}{q} + \frac{1}{q'}=1.
\end{equation}
\end{prop}
\begin{rem}\label{r4.2}
We do not require any positivity property for the solutions of \eqref{e4.4}.  We will come back to this point later on.
\end{rem}
\begin{proof}
Existence is based on the methods of compactness and monotonicity (see e.g. J.L. Lions [Lio69]) and on the following a priori estimate (only valid for $\varepsilon >0$ fixed).
%\vskip0.1in

Replacing $z_i$ by $Y_i-Y^u_i$ in \eqref{e4.5}, we obtain
\begin{equation}\label{e4.8}
\begin{split}
\frac{1}{2}\frac{d}{dt}\int_\Omega &(Y_i-Y^u_i)^2 dx + \int_\Omega [(\boldsymbol{v}\cdot\boldsymbol{\nabla}) Y_i](Y_i-Y^u_i)dx\\
&+\sum^N_{j=1}\int_\Omega\hat{a}_{ij}(\boldsymbol{Y})\boldsymbol{\nabla} Y_j\cdot\boldsymbol{\nabla} Y_idx +\varepsilon\int_\Omega|\boldsymbol{\nabla}\boldsymbol{Y}|^{q-2}|\boldsymbol{\nabla} Y_i|^2dx\\
&=\int_\Omega\omega_i(\theta,\boldsymbol{Y})(Y_i-Y^u_i)dx.
\end{split}
\end{equation}
The different terms in \eqref{e4.8} can be estimated by making use of \eqref{e2.28} and \eqref{e4.3}.  We find
\begin{equation*}
\begin{split}
\int_\Omega[(\boldsymbol{v}\cdot\boldsymbol{\nabla})Y_i](Y_i-Y^u_i)dx &=\frac{1}{2}\int_{\partial\Omega} (\boldsymbol{v}\cdot\boldsymbol{\nu})(Y_i-Y^u_i)^2d\Gamma - \frac{1}{2}\int_\Omega (Y_i-Y^u_i)^2 (\text{div }\boldsymbol{v})dx\\
&=\frac{1}{2}\int_{\Gamma_{h}}(Y_i - Y^u_i)^2d\Gamma\geq 0.
\end{split}
\end{equation*}
Also, 
\begin{equation*}
\begin{split}
&\bigg|\sum^N_{j=1}\int_\Omega\hat{a}_{ij}(\boldsymbol{Y})\boldsymbol{\nabla} Y_j\cdot\boldsymbol{\nabla} Y_idx\bigg|\leq K_1\sum^N_{j=1}\int_\Omega |\boldsymbol{\nabla} Y_j||\boldsymbol{\nabla} Y_i|dx,
\\
&\bigg|\int_\Omega\omega_i(\theta,\boldsymbol{Y})(Y_i-Y^u_i)dx\bigg|\leq K_2\int_\Omega |Y_i-Y^u_i|dx.
\end{split}
\end{equation*}
Combining the above inequalities with \eqref{e4.8} and adding for $i=1, \ldots, N,$ we conclude that
\begin{equation}\label{e4.9}
\begin{split}
\frac{1}{2}\frac{d}{dt}\sum^N_{i=1}\int_\Omega (Y_i-Y^u_i)^2 dx &+\varepsilon\int_{\Omega}|\boldsymbol{\nabla}\boldsymbol{Y}|^qdx\leq NK_1\int_\Omega|\boldsymbol{\nabla}\boldsymbol{Y}|^2dx\\
&+K_2\left\{\sum^N_{i=1}\int_\Omega |Y_i-Y^u_i|dx\right\}.
\end{split}
\end{equation}
This inequality readily yields, for fixed $\varepsilon$, a priori bounds of $Y_i-Y^u_i$ in $L^\infty(0,T; L^2(\Omega))$ and $L^q(0,T;W^{1,q}_{\Gamma_{0}}(\Omega))$. 

Next, combining these bounds and the weak formulation \eqref{e4.5} provide a priori bounds of $\frac{\partial Y_i}{\partial t}$ in $L^{q'}(0,T; (W^{1,q}_{\Gamma_{0}}(\Omega))')$.

These estimates allow us to show the existence of a solution of \eqref{e4.4}, \eqref{e4.4a}, \eqref{e2.22}$_1$, \eqref{e2.25} thanks to standard arguments:  introduction of a Galerkin approximation and passage to the limit by monotonicity and compactness (see e.g. \cite{Lio69}, p. 207).
\end{proof}
As already noticed, we did not require any positivity property for the $Y_j,$ when formulating the problem \eqref{e4.4}.  We conclude this section by showing that in fact such properties hold for the solutions that we have obtained.
%PROPOSITION 4.2
\begin{prop}\label{p4.2}
Under the assumptions of Theorem \ref{t2.1}, the solutions $Y_i$ of \eqref{e4.4}, \eqref{e4.4a}, \eqref{e2.22}$_1$, \eqref{e2.25} satisfy
\begin{equation}\label{e4.10}
0\leq Y_i(x,t)\leq 1, \text{ for }t\in [0,T]\text{ and a.e. } x\in\Omega,
\end{equation}
\begin{equation}\label{e4.11}
\sum^N_{j=1}Y_j(x,t)=1,\text{ for }t\in [0,T]\text{ and a.e. } x\in\Omega.
\end{equation}
\end{prop}

\begin{proof}  To derive the positivity, we set $z_i=-Y^-_i =\min(0,Y_i)\in W^{1,q}_{\Gamma_{0}}(\Omega)$ in \eqref{e4.5} and we find after some integrations by parts and upon using \eqref{e2.28}:
\begin{equation}\label{e4.12}
\begin{split}
\frac{1}{2}\frac{d}{dt}\int_\Omega (Y^-_i)^2 dx &+\frac{1}{2}\int_{\Gamma_{h}}(Y^-_i)^2d\Gamma - \sum^N_{j=1}\int_\Omega\hat{a}_{ij}(\boldsymbol{Y})\boldsymbol{\nabla} Y_j
\cdot\boldsymbol{\nabla} Y^-_idx\\
&+\varepsilon\int_\Omega |\boldsymbol{\nabla}\boldsymbol{Y}|^{q-2}|\boldsymbol{\nabla} Y^-_i|^2 dx =-\int_\Omega\omega_i(\theta,\boldsymbol{Y})Y^-_idx.
\end{split}
\end{equation}
Now at each point $(x,t)$ such that $Y_i(x,t)\leq 0,$ the definition \eqref{e4.1} of $\hat{a}_{ij}$ together with the assumptions \eqref{e2.10}, \eqref{e2.11} guarantee that
\begin{equation*}
-\sum^N_{j=1}\hat{a}_{ij}(Y_1, \ldots, Y_N)\boldsymbol{\nabla} Y_j\cdot\boldsymbol{\nabla} Y^-_i = \xi \left(\sum^N_{\ell=1}|Y_\ell |\right) b^0_i(Y^+_1,\ldots , Y^+_N)|\boldsymbol{\nabla} Y^-_i|^2\geq 0.
\end{equation*}
while, the definition \eqref{e4.2} of the extended $\omega_i$ together with the assumptions \eqref{e2.15}, \eqref{e2.16} provide that:
\begin{equation*}
\omega_i(\theta, Y_1,\ldots, Y_N) Y^-_i=\alpha_i(\theta^+, \psi(Y_1),\ldots ,\psi(Y_N))Y^-_i \geq 0.
\end{equation*}
Therefore we infer from  \eqref{e4.12} that
\begin{equation*}
\frac{1}{2}\frac{d}{dt}\int_\Omega (Y^-_i)^2 dx\leq 0,
\end{equation*}
which on integrating yields, due to the positivity of the initial data (cf. \eqref{e2.27}):
\begin{equation}\label{e4.15}
Y_i(x,t)\geq 0\text{ for }t\in[0,T]\text{ and a.e. } x\in\Omega .
\end{equation}

Consequently, recalling  again the definition \eqref{e4.1}, we now have 
$$\hat{a}_{ij}(Y_1, \ldots, Y_N) = \xi (\sum^N_{\ell=1}Y_\ell)\enspace a_{ij} (Y_1,\ldots ,Y_N)$$ 
so that
\begin{equation}\label{e4.16aa}
\sum^N_{i=1}\boldsymbol{\hat{F}}_i=-\xi (\sum^N_{\ell=1}Y_\ell) \left[\sum^N_{i,j=1}a_{ij}(Y_1,\ldots ,Y_N)\boldsymbol{\nabla} Y_j\right]=0,
\end{equation}
thanks to \eqref{e2.9}.

We aim now to  derive \eqref{e4.11}. Let us add the equations \eqref{e4.4} for $i=1,...,N$. By  \eqref{e4.16aa}, the sum of the fluxes vanishes while the property \eqref{e2.17} still holds for the extended non linearities $\omega_i$. Consequently, $\mathcal{U} = \sum^N_{i=1} Y_i$ satisfies:
\begin{equation}\label{e4.16ab}
\frac{\partial \mathcal{U}}{\partial t} +(\boldsymbol{v}\cdot\boldsymbol{\nabla})\mathcal{U} -\varepsilon\boldsymbol{\nabla}\cdot \left[|\boldsymbol{\nabla} Y|^{q-2}\boldsymbol{\nabla}
\mathcal{U}\right]=0.
\end{equation}
In view of the boundary conditions for the $Y_i$, we have $\mathcal{U}=1\text{ on }\Gamma_0$ while, on $\Gamma_h\cup\Gamma_\ell$, by adding the conditions \eqref{e4.4a} and using again \eqref{e4.16aa}, we see that:
\begin{equation*}
\varepsilon |\nabla Y|^{q-2}\frac{\partial\mathcal{U}}{\partial\boldsymbol{\nu}} = 0
\end{equation*}
which guarantees that $\frac{\partial\mathcal{U}}{\partial\boldsymbol{\nu}} = 0\text{ on }\Gamma_h\cup\Gamma_\ell$. This gives readily \eqref{e4.11} since the linear equation \eqref{e4.16ab} possesses a unique solution satisfying the above boundary conditions together with $\mathcal{U}=1$ at $t=0.$

This concludes the proof of Proposition \ref{p4.2}, since \eqref{e4.11} together with the positivity of the $Y_i$ provide that
$Y_i(x,t)\leq 1\text{ for }t\in [0,T]\text{ and a.e. }x\in\Omega$. 
\end{proof}

It is worth noting that, since $Y_i(x,t)\geq 0$ and $\sum^N_{i=1}Y_i(x,t)=1$ a.e., we have $\hat{a}_{ij}(\boldsymbol{Y})=a_{ij}(\boldsymbol{Y})$  so that 
\eqref{e4.4} now reads 
 \begin{equation}\label{e4.16}
\begin{split}
\frac{\partial Y_i}{\partial t}+(\boldsymbol{v}\cdot\boldsymbol{\nabla})Y_i&-\sum^N_{j=1}\boldsymbol{\nabla}\cdot (a_{ij}(Y_1,\ldots ,Y_N)\boldsymbol{\nabla} Y_j)\\
&-\varepsilon\boldsymbol{\nabla}\cdot (|\boldsymbol{\nabla}\boldsymbol{Y}|^{q-2}\boldsymbol{\nabla} Y_i)=\omega_i(\theta, Y_1,\ldots ,Y_N).
\end{split}
\end{equation}
Also, the fluxes in \eqref{e4.16} are indeed the solutions of the Stefan Maxwell equations (see Theorem 3.2).

\vskip0.1in

\subsection{The energy equation}\label{ss4.2}\hspace*{\fill} \\
We aim now to prove Theorem \ref{t2.1}.  The solution of \eqref{e2.4}, \eqref{e2.22}, \eqref{e2.25} will be obtained by taking the limit $\varepsilon\rightarrow 0$ in \eqref{e4.16}. For that purpose we need a priori estimates independent of $\varepsilon$ for the solutions of this problem (we still omit to denote the dependence of $Y_i$ on  $\varepsilon$ to make notations simpler).

As mentioned in the introduction, for the original problem \eqref{e2.4}, assuming that the $Y_i\enspace (X_i)$ do not vanish, the natural Gibbs energy equation is obtained by multiplying equations \eqref{e2.4} by $\mu_i=\frac{1}{M_i}\log X_i$ and adding for $i=1,\ldots, N.$  More precisely, in view of the boundary conditions, we should multiply \eqref{e2.4} by
\begin{equation}\label{e4.16a}
\mu_i-\mu^u_i = \frac{1}{M_i}(\log X_i - \log X^u_i),
\end{equation}
with
\begin{equation}\label{e4.16b}
X^u_i =\frac{Y^u_i}{M_iY^u_M},\enspace Y^u_M =\sum^N_{i=1}\frac{Y^u_i}{M_i}.
\end{equation}
The $\mu_i$ (resp. $\mu^u_i$) can be expressed in terms of the $Y_j$ (resp. $Y^u_j$) by using the $\boldsymbol{X}-\boldsymbol{Y}$ relations \eqref{e3.12}:
\begin{equation}\label{e4.16ba}
\mu_i =\frac{1}{M_i}\log \frac{Y_i}{M_iY_M} = \frac{1}{M_i}\log\frac{Z_i}{\sum^N_{j=1}Z_j} \text{ with }   Z_i=Y_i/M_i.
\end{equation}
(resp. $Z^u_i=Y^u_i/M_i$).

\vskip0.1in

From a mathematical point of view, since the $Y_i$ might vanish, we introduce a parameter $\eta >0,$ and, instead of $\mu_i$, consider: 
\begin{equation}\label{e4.16bb}
\mu^\eta_i=\frac{1}{M_i}\log\frac{Z^\eta_i}{\sum^N_{j=1}Z^\eta_j}, \quad Z^\eta_i=\frac{Y_i+\eta}{M_i},
\end{equation}
with a similar definition for $\mu^{u,\eta}_i$.

We multiply the equations \eqref{e4.16} by $\mu^\eta_i -\mu^{u,\eta}_i$, integrate over $\Omega$ and add for $i=1,\ldots ,N$. For the term involving the time derivatives, we observe that
\begin{equation}\label{e4.16bc}
\mu^\eta_i - \mu^{u,\eta}_i =\frac{\partial}{\partial Y_i} g^\eta (Y_1,\ldots Y_N),
\end{equation}
where
\begin{equation}\label{e4.16bd}
g^\eta(Y_1,\ldots ,Y_N) = \sum^N_{j=1}Z^\eta_j \left[\log\frac{Z^\eta_j}{\sum^N_{\ell=1}Z^\eta_\ell} -\log\frac{Z^{u,\eta}_j}{\sum^N_{\ell=1}Z^{u,\eta}_\ell}\right].
\end{equation}
Hence
\begin{equation*}
\sum^N_{i=1}\bigg<\frac{\partial Y_i}{\partial t}, \mu^\eta_i -\mu^{u,\eta}_i\bigg> =\sum^N_{i=1}\int_\Omega\frac{\partial g^\eta}{\partial Y_i}
\frac{\partial Y_i}{\partial t}dx,
\end{equation*}
and
\begin{equation}\label{e4.16c}
\sum^N_{i=1}<\frac{\partial Y_i}{\partial t}, \mu^\eta_i -\mu^{u,\eta}_i> =\frac{d}{dt}\int_\Omega g^\eta(Y_1,\ldots Y_N)dx.
\end{equation}
Note that $g^\eta$ is bounded independently of $\eta \in ]0,1[$ for bounded values of $Z^\eta_j$ ($0 \leq Z^\eta_j \leq 2/M_j$ in our case).  Note also that \eqref{e4.16c} proven as if the $Y_i$ were smooth can be proven by approximation for the actual functions $Y_i,$ observing that
\begin{equation}\label{e4.16d}
\begin{cases}
\frac{\partial Y_i}{\partial t}\in L^{q'}(0,T; (W^{1,q}_{\Gamma_{0}}(\Omega))' )\text{ and }\\
&\\
\mu^\eta_i - \mu^{u,\eta}_i\in L^q(0,T; W^{1,q}_{\Gamma_{0}}(\Omega)).
\end{cases}
\end{equation}
A similar remark applies to several of the following terms.

Next, concerning the contribution of the convective terms to the energy equation,   we write
\begin{equation*}
\begin{split}
\sum^N_{i=1}\int_\Omega [(\boldsymbol{v}\cdot&\boldsymbol{\nabla})Y_i] (\mu^\eta_i-\mu^{u,\eta}_i)=\sum^N_{j=1}\int_\Omega v_j\frac{\partial}{\partial x_j}g^\eta(\boldsymbol{Y})dx\\
&=\int_{\partial\Omega} (\boldsymbol{v}\cdot\boldsymbol{\nu})g^\eta(\boldsymbol{Y})d\Gamma -\int_\Omega \textrm{ div }\boldsymbol{v}\, g^\eta(\boldsymbol{Y})dx=\int_{\Gamma_{h}}g^\eta(\boldsymbol{Y})d\Gamma,
\end{split}
\end{equation*}
as $\text{ div }\boldsymbol{v} = 0$, $g^\eta(\boldsymbol{Y}) = 0$ at $x_n=0$ and in view of the boundary conditions for $\boldsymbol{v}$.

\vskip0.1in

Performing also some integration by parts in the integrals related to the diffusive terms and nonlinear Laplacian, our energy equation reads
\begin{equation}\label{e4.16ff}
\begin{split}
 \frac{d}{dt} \int_\Omega g^\eta & (\boldsymbol{Y})dx + \int_{\Gamma_{h}}g^\eta(\boldsymbol{Y})d\Gamma +\sum^N_{i,j=1}\int_\Omega a_{ij}(\boldsymbol{Y})
\boldsymbol{\nabla} Y_j\cdot\boldsymbol{\nabla} \mu^\eta_idx  \\
&+\sum^N_{i=1}\varepsilon\int_\Omega |\boldsymbol{\nabla}\boldsymbol{Y}|^{q-2}\boldsymbol{\nabla} Y_i\cdot\boldsymbol{\nabla} \mu^\eta_idx=\sum^N_{i=1}\int_\Omega\omega_i(\theta,\boldsymbol{Y})(\mu^\eta_i-\mu^{u,\eta}_i)dx.
\end{split}
\end{equation}

\vskip0.1in

We now aim  and to pass to the limit $\eta\rightarrow 0$ in \eqref{e4.16ff}. We plan in this way to obtain estimates independent of $\varepsilon$ for the $Y_i=Y_{i,\varepsilon}$. Note that $\mu^\eta_i$ is singular when $\eta \rightarrow 0$ if $Y_i = 0$ but as we will see below this singularity is usually absorbed by other factors.

\subsection{Passage to the limit $\eta\rightarrow 0$}\hspace*{\fill} \\
We first observe that the right hand-side of \eqref{e4.16ff} is bounded from above independently of $\eta \in ]0,1[$ and $\varepsilon >0$. Indeed recalling the decomposition \eqref{e2.15} of $\omega_i$, we have:
\begin{equation}\label{e4.16h}
\omega_i(\theta,\boldsymbol{Y})(\mu^\eta_i-\mu^{u,\eta}_i)=\alpha_i(\theta,\boldsymbol{Y})\mu^\eta_i - \beta_i(\theta,\boldsymbol{Y})Y_i\mu^\eta_i - \omega_i(\theta,\boldsymbol{Y})\mu^{u,\eta}_i.
\end{equation}
Here, the assumption \eqref{e2.16} together with the definition \eqref{e4.16bb} of 
$\mu^{\eta}_i$ guarantee that 
$\alpha_i(\theta,\boldsymbol{Y})\mu^\eta_i \leq 0$
 while, by \eqref{e2.18}, 
$\beta_i$ and $\omega_i$ are bounded functions. Next $\mu^{\eta}_i$ reads
\begin{equation}\label{e4.16fk}
\mu^\eta_i=\frac{1}{M_i}\log\frac{Y_i+\eta}{M_iY^\eta_M},\enspace\text{ with} \enspace
Y^\eta_M=\sum^N_{j=1}Z^\eta_j.
\end{equation}
Since $Y_i\geq 0$ and $\sum^N_{j=1}Y_j=1$ a.e., the lower bound \eqref{e3.14} holds true and $Y^\eta_M$ is bounded from below:
$$Y^\eta_M \geq Y_M \geq \frac{1}{\overline{M}}.$$
Consequently  the quantities $Y_i\mu^\eta_i$
are bounded independently of $0<\eta<1$ and $\varepsilon$ while, since all the $Y^u_i$ are strictly positive,
the constants $\mu^{u,\eta}_i$ are also bounded. 

Also, recall that $g^\eta(\boldsymbol{Y})$ given by \eqref{e4.16bd} is bounded independently of $\eta \in ]0,1[$ and $\varepsilon >0$. Therefore, coming back to \eqref{e4.16ff} that we integrate on $(0,T)$, we conclude that there exists a constant $c_2$ independent of $0<\eta<1$ and $\varepsilon$ such that:
\begin{equation}\label{e4.16fg}
\sum^N_{i,j=1}\int_0^T\int_\Omega a_{ij}(\boldsymbol{Y})
\boldsymbol{\nabla} Y_j\cdot\boldsymbol{\nabla} \mu^\eta_idxds  +\sum^N_{i=1}\varepsilon\int_0^T\int_\Omega |\boldsymbol{\nabla}\boldsymbol{Y}|^{q-2}\boldsymbol{\nabla} Y_i\cdot\boldsymbol{\nabla} \mu^\eta_idxds \leq c_2.
\end{equation}

We now aim to take the limit $\eta\rightarrow 0$ in the two terms in the right hand-side of \eqref{e4.16fg}. It follows from \eqref{e4.16fk} that
\begin{equation}\label{e4.16fj}
\boldsymbol{\nabla}\mu^\eta_i = \frac{1}{M_i}\frac{\boldsymbol{\nabla} Y_i}{Y_i+\eta} -\frac{1}{M_i}\frac{\boldsymbol{\nabla} Y_M}{Y^\eta_M}.
\end{equation}
Hence, for the first term in \eqref{e4.16fg}, we can write:
\begin{equation}\label{e4.16e}
\begin{split}
\int_0^T &  \int_\Omega a_{ij}(\boldsymbol{Y})
\boldsymbol{\nabla} Y_j\cdot\boldsymbol{\nabla} \mu^\eta_idxds = \\
&\int_0^T\int_\Omega \left[ \frac{a_{ij}(\boldsymbol{Y})}{M_i(Y_i+\eta)}\boldsymbol{\nabla} Y_j\cdot\boldsymbol{\nabla} Y_i - \frac{a_{ij}(\boldsymbol{Y})}{M_iY^\eta_M}\boldsymbol{\nabla} Y_j\cdot\boldsymbol{\nabla} Y_M\right]dxds.
\end{split}
\end{equation}
We observe that all the integrands vanish a.e. when $Y_i=0$ since either $i\not= j$ and, by \eqref{e2.10}, $a_{ij} = 0$, or $i=j$ and $\boldsymbol{\nabla} Y_i$ vanishes (a.e.). Next, we can easily pass to the limit $\eta\rightarrow 0$ in the second integral of the right hand-side of \eqref{e4.16e} by using Lebesgue's theorem since  $Y^\eta_M$ is bounded from above and from below by positive constants (independent of $\eta$) and converges pointwise to $Y_M$ as $\eta\rightarrow 0$, while the other functions are integrable since $\boldsymbol{\nabla} Y_j\in L^q(0,T; L^q(\Omega)^n)$. Hence we obtain
\begin{equation*}
\begin{split}
\int_0^T\int_\Omega \frac{a_{ij}(\boldsymbol{Y})}{M_iY^\eta_M}\boldsymbol{\nabla} Y_j\cdot\boldsymbol{\nabla} Y_Mdxds \rightarrow \int_0^T\int_\Omega \mathbbm{1}_{\left\{Y_i>0\right\}}\frac{a_{ij}(\boldsymbol{Y})}{M_iY_M}\boldsymbol{\nabla} Y_j\cdot\boldsymbol{\nabla} Y_Mdxds.
\end{split}
\end{equation*}
For the first integral in \eqref{e4.16e}, we will use the properties \eqref{e2.10}, \eqref{e2.11} and distinguish the cases  $i\not= j$ and $i=j$. If $i\not= j$, in view of \eqref{e2.10},  we observe that :
\begin{equation*}
\frac{a_{ij}(\boldsymbol{Y})}{M_i(Y_i+\eta)}\boldsymbol{\nabla} Y_j\cdot\boldsymbol{\nabla} Y_i=\mathbbm{1}_{\left\{Y_i>0\right\}} \enspace \frac{b_{ij}(\boldsymbol{Y})}{M_i} \frac{Y_i}{Y_i+\eta}\boldsymbol{\nabla} Y_j\cdot\boldsymbol{\nabla} Y_i.
\end{equation*}
This quantity converges pointwise to
\begin{equation*}
\mathbbm{1}_{\left\{Y_i>0\right\}}\enspace \frac{b_{ij}(\boldsymbol{Y})}{M_i}\boldsymbol{\nabla} Y_j\cdot\boldsymbol{\nabla} Y_i
=\mathbbm{1}_{\left\{Y_i>0\right\}}\enspace \frac{a_{ij}(\boldsymbol{Y})}{M_iY_i}\boldsymbol{\nabla} Y_j\cdot\boldsymbol{\nabla} Y_i,
\end{equation*}
and the corresponding integrals converge by Lebesgue's theorem. Next if $i=j$, by \eqref{e2.11},
\begin{equation*}
\frac{a_{ii}(\boldsymbol{Y})}{M_i(Y_i+\eta)}|\boldsymbol{\nabla} Y_i|^2 
= \frac{b^0_i(\boldsymbol{Y})}{M_i(Y_i+\eta)}|\boldsymbol{\nabla} Y_i|^2 + \frac{b^1_i (\boldsymbol{Y})}{M_i} \frac{Y_i}{Y_i+\eta}|\boldsymbol{\nabla} Y_i|^2.
\end{equation*}
Similarly to above, we have
\begin{equation*}
\frac{b^1_i (\boldsymbol{Y})}{M_i} \frac{Y_i}{Y_i+\eta}|\boldsymbol{\nabla} Y_i|^2
\rightarrow\mathbbm{1}_{\left\{Y_i>0\right\}}\enspace \frac{b^1_i (\boldsymbol{Y})}{M_i}|\boldsymbol{\nabla} Y_i|^2,
\end{equation*}
hence the convergence of the integrals.  For the terms involving $b^0_i,$ we observe that $b^0_i(\boldsymbol{Y})\geq 0$ so that we can pass to the lower limit  by Fatou's Lemma and obtain:
\begin{equation}\label{e4.23a}
\begin{split}
\int^T_0\int_\Omega\mathbbm{1}_{\left\{Y_i>0\right\}}\enspace \frac{b^0_i(\boldsymbol{Y})}{M_i}\frac{|\boldsymbol{\nabla} Y_i|^2}{Y_i} dxds&=\\
&\leq \liminf_{\eta\rightarrow 0}\int^T_0\int_\Omega\mathbbm{1}_{\left\{Y_i>0\right\}}\enspace \frac{b^0_i(\boldsymbol{Y})}{M_i}\frac{|\boldsymbol{\nabla} Y_i|^2}{Y_i+\eta}dxds.
\end{split}
\end{equation}
In the context of the final a priori estimates below (collected estimates), \eqref{e4.23a} implies that its left hand-side is indeed integrable.

\vskip0.1in

Using again \eqref{e4.16fj}, the second  term in \eqref{e4.16fg} reads:
\begin{equation}\label{e4.16f}
\begin{split}
\sum^N_{i=1} \varepsilon & \int_0^T  \int_\Omega |\boldsymbol{\nabla}\boldsymbol{Y}|^{q-2}\boldsymbol{\nabla} Y_i\cdot\boldsymbol{\nabla} \mu^\eta_idxds=\\
&\sum^N_{i=1}\varepsilon \int_0^T\int_\Omega|\boldsymbol{\nabla} \boldsymbol{Y}|^{q-2}\frac{|\boldsymbol{\nabla} Y_i|^2}{M_i(Y_i+\eta)}dxds -\varepsilon \int_0^T\int_\Omega  |\boldsymbol{\nabla}\boldsymbol{Y}|^{q-2}\frac{|\boldsymbol{\nabla} Y_M|^2}{Y^\eta_M} dxds.
\end{split}
\end{equation}
We can easily pass to the limit $\eta\rightarrow 0$ by using Lebesgue's theorem in the second term. Concerning the first one, the integrand is positive, so we can take the lower limit using Fatou's Lemma so that :
\begin{equation}\label{e4.23b}
\begin{split}
\sum^N_{i=1} \varepsilon \int^T_0\int_\Omega|\boldsymbol{\nabla} \boldsymbol{Y}|^{q-2}&\mathbbm{1}_{\left\{Y_i>0\right\}}\enspace\frac{1}{M_i}\frac{|\boldsymbol{\nabla} Y_i|^2}{Y_i}dxds\\
&\leq\liminf_{\eta\rightarrow 0}\sum^N_{i=1} \varepsilon \int^T_0\int_\Omega |\boldsymbol{\nabla}\boldsymbol{Y}|^{q-2} \mathbbm{1}_{\left\{Y_i>0\right\}}\frac{|\boldsymbol{\nabla} Y_i|^2}{M_i(Y_i+\eta)}dxds.
\end{split}
\end{equation}
As for \eqref{e4.23a}, this eventually implies that the left hand side of \eqref{e4.23b} is integrable.

\vskip0.1in

By collecting all the results above we can pass to the lower limit in \eqref{e4.16fg} as $\eta \rightarrow 0$ and we obtain that:
\begin{equation}\label{e4.16i}
\begin{split}
\int_0^T\int_\Omega\ & \sum^N_{i,j=1}\mathbbm{1}_{\left\{Y_i>0\right\}}a_{ij}(\boldsymbol{Y})\boldsymbol{\nabla} Y_j\cdot\boldsymbol{\nabla}\mu_idxds \\
&+\varepsilon\int_0^T\int_\Omega |\boldsymbol{\nabla}\boldsymbol{Y}|^{q-2}\left[\sum^N_{i=1}\frac{|\boldsymbol{\nabla} Y_i|^2}{M_iY_i}\mathbbm{1}_{\left\{Y_i>0\right\}} -
\frac{|\boldsymbol{\nabla} Y_M|^2}{Y_M}\right]dxds \leq c_2.
\end{split}
\end{equation}

\vskip0.1in

\subsection{Passage to the limit $\varepsilon\rightarrow 0$}\label{ss4.3}\hspace*{\fill} \\
We first derive from \eqref{e4.16i} some estimates of the $Y_i=Y_{i,\varepsilon}$ that are independent of $\varepsilon$. Recalling \eqref{e2.8}, we observe that the first term in \eqref{e4.16i} is equal to
\begin{equation*}
-\sum^N_{i=1}\int_0^T\int_\Omega \textbf{F}_i\cdot\boldsymbol{\nabla} \mu_i\mathbbm{1}_{\left\{Y_i>0\right\}}dxds,
\end{equation*}
and thanks to \eqref{e2.12} it is bounded from below by
\begin{equation*}
c_1\int_0^T\int_\Omega|\boldsymbol{\nabla}\boldsymbol{Y}|^2 dxds.
\end{equation*}
Also we observe that the second term in \eqref{e4.16i} is positive because
\begin{equation*}
\begin{split}
\frac{|\boldsymbol{\nabla} Y_M|^2}{Y_M} &= \frac{1}{Y_M}\left|\sum^N_{j=1}\frac{\boldsymbol{\nabla} Y_j}{M_j}\mathbbm{1}_{\left\{Y_j>0\right\}}\right|^2\\
&\leq\frac{1}{Y_M}\left(\sum^N_{j=1}\frac{|\boldsymbol{\nabla} Y_j|^2}{M_jY_j}\mathbbm{1}_{\left\{Y_j>0\right\}}\right)\enspace
\left(\sum^N_{j=1}\frac{Y_j}{M_j}\right)\\
&\leq\sum^N_{j=1}\frac{|\boldsymbol{\nabla} Y_j|^2}{M_jY_j}\mathbbm{1}_{\left\{Y_j>0\right\}}.
\end{split}
\end{equation*}
With this \eqref{e4.16i} yields
\begin{equation}\label{e4.16j}
c_1\int_0^T\int_\Omega |\boldsymbol{\nabla}\boldsymbol{Y}|^2dxds\leq c_2,
\end{equation}
where $c_1$ and $c_2$ are independent of $\varepsilon$, so that:
\begin{equation}\label{e4.31}
\int_0^T\int_\Omega |\boldsymbol{\nabla}\boldsymbol{Y}|^2dxds \text{ is bounded independently of }\varepsilon.
\end{equation}
Going back to \eqref{e4.9}, \eqref{e4.31} together with \eqref{e4.10} guarantee that
\begin{equation}\label{e4.32c}
\varepsilon\int^{T}_0\int_\Omega |\boldsymbol{\nabla} \boldsymbol{Y}|^qdxds\text{ is bounded independently of }\varepsilon.
\end{equation}
\par
 Thanks to the estimates \eqref{e4.7}, \eqref{e4.31} and \eqref{e4.32c}, we can take the limit $\varepsilon\rightarrow 0$ in \eqref{e4.16} and obtain a weak solution of \eqref{e2.4}.  The details are standard.  This concludes the proof of Theorem \ref{t2.1}.

\section{The Full System}\label{s5}
\par
In this section, we investigate problem \eqref{e2.1}-\eqref{e2.4} and aim to prove Theorem \ref{t2.2}.  The $\boldsymbol{Y}-$ system is now coupled with the equations for $\boldsymbol{v}$ and $\theta.$  Clearly, in comparison with our study in Section \ref{s4}, the main new point is to derive estimates like \eqref{e2.28}, \eqref{e2.29} for $\boldsymbol{v}$ and $\theta.$  As in Section \ref{s4}, we derive such estimates for an appropriate modified problem, and then we take the limit $\varepsilon\rightarrow 0.$
\par
The $\boldsymbol{Y}-$ equations \eqref{e2.4} are modified as in Section \ref{s4} by considering \eqref{e4.4}.  Now, this system is coupled with
\begin{equation}\label{e5.1}
\frac{\partial\boldsymbol{v}}{\partial t} +(\boldsymbol{v}\cdot\boldsymbol{\nabla})-Pr\Delta\boldsymbol{v} +\boldsymbol{\nabla} p=\boldsymbol{e}_n\sigma\theta,
\end{equation}
\begin{equation}\label{e5.2}
\text{div }\boldsymbol{v} =0,
\end{equation}
\begin{equation}\label{e5.3}
\frac{\partial\theta}{\partial t}+(\boldsymbol{v}\cdot\boldsymbol{\nabla})\theta -\Delta\theta =-\sum^N_{i=1}h_i\omega_i(\theta, Y_1, \ldots , Y_N).
\end{equation}
As before, we show the existence of a solution of
\eqref{e5.1}-\eqref{e5.3},  \eqref{e4.4} (together with the appropriate initial and boundary conditions) thanks to the methods of compactness and monotonicity.  The useful a priori estimates derived hereafter are based on the fact that the $\omega_i$ are bounded independently of $\boldsymbol{Y}$ and $\theta,$ thanks to \eqref{e4.3}.  We first multiply \eqref{e5.3} by $\theta$ and integrate over $\Omega.$  This provides
\begin{equation*}
\begin{split}
\frac{1}{2}\frac{d}{dt}\int_\Omega \theta^2 dx +&\frac{1}{2} \int_{\Gamma_{h}} \theta^2 d\Gamma +\int_\Omega |\boldsymbol{\nabla}\theta |^2dx =- \int_\Omega \left(\sum^N_{i=1}h_i\omega_i(\theta, \boldsymbol{Y})\right)\theta dx\\
&\leq (\text{due to }\eqref{e4.3})\\
&\leq K_2 \left( \sum^N_{i=1}h_i \right) \int_\Omega |\theta| dx,
\end{split}
\end{equation*}
which yields readily that
\begin{equation}\label{e5.4}
\theta\text{ is bounded in }L^\infty (0,T; L^2(\Omega))\cap L^2(0,T; H^1(\Omega)).
\end{equation}
\par
Now, the right-hand side of \eqref{e5.1} is bounded in $L^\infty(0,T, L^2(\Omega)).$  Classical estimates for the two and three dimensional Navier-Stokes equations (see e.g. \cite{Tem77}) provide that
\begin{equation}\label{e5.5}
\boldsymbol{v}\text{ is bounded in }L^\infty (0,T, L^2(\Omega)^n)\cap L^2(0,T; H^1(\Omega)^n).
\end{equation}
In particular, the estimates corresponding to \eqref{e2.28} and \eqref{e2.29} have now been derived for $\boldsymbol{v}$ and $\theta.$  We can then proceed as in the proof of Proposition \ref{p4.2} and show that
\begin{equation}\label{e5.6}
\boldsymbol{Y}\text{ is bounded in }L^\infty (0,T, L^2 (\Omega)^N)\cap L^q(0,T; W^{1,q}(\Omega)^N).
\end{equation}
It follows easily from \eqref{e5.4}-\eqref{e5.6} that the system consisting of \eqref{e4.4} and \eqref{e5.1}-\eqref{e5.3} supplemented with the boundary conditions \eqref{e2.20}, \eqref{e2.21}, \eqref{e2.22}$_1$, \eqref{e4.4a} and the initial conditions \eqref{e2.24}, \eqref{e2.25}  possesses a solution $(\boldsymbol{v}, \theta,\boldsymbol{Y}).$  Also, \eqref{e4.10}, \eqref{e4.11} hold for the same reasons as before. Furthermore we have
\begin{equation}\label{e5.7}
\theta(x,t)\geq 0\text{ for }t\in[0,T]\text{ and a.e. }x\in\Omega .
\end{equation}
Indeed, multiplying \eqref{e5.3} by $-\theta^-=\min (0,\theta)$ and integrating over $\Omega,$ we obtain
\begin{equation}\label{e5.8}
\begin{split}
\frac{1}{2}\frac{d}{dt}\int_\Omega (\theta^-)^2 dx +&\frac{1}{2}\int_{\Gamma_{h}} (\theta^-)^2d\Gamma +\int_\Omega |\boldsymbol{\boldsymbol{\nabla}}\theta^-|^2dx\\
&=\int_\Omega\left(\sum^N_{i=1}h_i\omega_i(\theta, \boldsymbol{Y})\right)\theta^-dx.
\end{split}
\end{equation}
Due to the definition \eqref{e4.2} of $\omega_i$ and \eqref{e2.19}, at each point $(x,t)$ such that $\theta (x,t)\leq 0,$ we have
\begin{equation*}
\sum^N_{i=1}h_i\omega_i(\theta, \boldsymbol{Y})=\sum^N_{i=1}h_i\omega_i(0, \boldsymbol{Y})\leq 0,
\end{equation*}
and therefore
\begin{equation}\label{e5.9}
\int_\Omega \left(\sum^N_{i=1}h_i\omega_i(\theta, \boldsymbol{Y})\right)\theta^-dx\leq 0.
\end{equation}
Combining \eqref{e5.9} with \eqref{e5.8} enables us to show that
\begin{equation*}
\frac{d}{dt}\int_\Omega (\theta^-)^2dx\leq 0,
\end{equation*}
and thus to obtain \eqref{e5.7} since $\theta_0(x)\geq 0$ for almost every $x\in\Omega.$
\par
The modified system \eqref{e5.1}-\eqref{e5.3}, \eqref{e4.4}, \eqref{e4.4a} depends on a parameter $\varepsilon >0$ (in \eqref{e4.4} and \eqref{e4.4a}) and as in Section \ref{s4} we need to take the limit $\varepsilon\rightarrow 0.$  The estimates \eqref{e5.4} and \eqref{e5.5} are independent of $\varepsilon$ while $\boldsymbol{Y}$ can be estimated independently of $\varepsilon$ exactly as in Section \ref{s4}.  Based on these estimates it is easy to see that we can take the limit $\varepsilon\rightarrow 0$ in \eqref{e5.1}-\eqref{e5.3}, \eqref{e4.4}, \eqref{e4.4a} and obtain a weak solution of \eqref{e2.1}-\eqref{e2.4}.  The passage to the limit in the $\boldsymbol{Y}-$ equations is done as in Section \ref{s4}; the passage to the limit in the $\boldsymbol{v}$ and $\theta$ equations is standard.  The details are left to the reader.

%\rem{Remark}\label{r1.3}
%\bigskip
\vskip0.1in

\section*{Acknowledgments}  This work was partially
supported by the National Science Foundation under the grant
NSF-DMS-1206438, and by the Research Fund of Indiana University.

\end{document}